\documentclass[11pt]{article}
\usepackage{mathtools,amsmath,amsthm,amssymb,mathrsfs,amsfonts}
\usepackage{epsfig}
\usepackage{leftidx}
\usepackage{color,epstopdf}
\usepackage{cite}
\usepackage{graphicx,wrapfig,tabularx}
\usepackage{flafter}
\usepackage{fancyhdr}
\usepackage{stmaryrd}
\usepackage{graphicx,subfigure}
\usepackage{multicol,multirow,booktabs}
\usepackage{booktabs,threeparttable}
\usepackage[center]{caption2}
\usepackage{epstopdf}
\usepackage{multirow}
\usepackage{enumerate}
\usepackage{enumitem}
\usepackage{algpseudocode,algorithm,algorithmicx}

\newtheorem{theorem}{Theorem}[section]
\newtheorem{lemma}{Lemma}[section]

\newtheorem{example}{Example}

\def\rmk{{\mathrm{k}}}
\def\ck{{\scriptstyle{K}}}

\newcommand{\myvec}[1]{\mathbf{#1}}

\newcommand{\zd}{\,\mathrm{d}}

\newcommand{\abs}[1]{\left|#1\right|}
\newcommand{\absb}[1]{\big|#1\big|}

\newcommand{\bra}[1]{\left(#1\right)}
\newcommand{\brab}[1]{\big(#1\big)}
\newcommand{\braB}[1]{\Big(#1\Big)}

\newcommand{\kbra}[1]{\left[#1\right]}
\newcommand{\kbrab}[1]{\big[#1\big]}

\newcommand{\myinner}[1]{\left\langle#1\right\rangle}

\newcommand{\myinnerb}[1]{\big\langle#1\big\rangle}

\newcommand{\mynorm}[1]{\left\|#1\right\|}
\newcommand{\mynormb}[1]{\big\|#1\big\|}

\def\lan#1{\textcolor{blue}{#1}}

\title{Stability analysis of consistent splitting implicit-explicit multistep methods up to \lan{ninth-order} accuracy for incompressible flows}
\author{Yuanyuan Kang\thanks{School of Mathematics, Nanjing University of Aeronautics and Astronautics, Nanjing 211106, China. Email: kangyy0101@163.com.}
			\and
	Hong-lin Liao\thanks{ORCID 0000-0003-0777-6832. School of Mathematics,
    Nanjing University of Aeronautics and Astronautics,
    Nanjing 211106, China; Key Laboratory of Mathematical Modeling
    and High Performance Computing of Air Vehicles (NUAA), MIIT, Nanjing 211106, China. 
    Email: liaohl@nuaa.edu.cn.    This author's work is supported in part by National Natural Science Foundation of China under grant 12471383, National Natural Science Foundation of Jiangsu Province under grant BK20252027,
    and Ministry of Education Key Laboratory of NSLSCS under grant 202501.}
\and Guidong Liu\thanks{School of Mathematics, Nanjing Audit University, Nanjing 211815, China. 
	Email: liugd@nau.edu.cn}
}

\begin{document}
  
\maketitle

\begin{abstract}
This work presents a concise, unified stability theory of high-order decoupled \lan{implicit-explicit linear multistep (IELM)} methods based on the well-known consistent splitting technique for the incompressible Navier-Stokes equation (INSE).
With the help of the recent semi-generating function approach and the global discrete energy analysis, one can establish the unconditional stability of a consistent splitting IELM method with respect to the $\ell^{\infty}(H^1)\cap \ell^{2}(H^2)$ norm if the associated implicit-explicit controllability intensity is larger than $\sqrt{2}/2$, a constant determined by the Stokes pressure estimate. It is shown that the $\beta$-parameterized GBDF-$\rmk$ ($2\le \rmk\le5$) schemes and $\gamma$-parameterized SIELM-$\rmk$ ($2\le \rmk\le9$) schemes can fulfill this requirement of implicit-explicit controllability intensity by choosing proper parameters so that they can theoretically maintain the unconditional stability of the associated consistent splitting IELM methods. Numerical experiments are also included to support our theory.
	\\[1ex]   
	\\[1ex]
 \end{abstract}


\section{Introduction}\label{sec: introduction}
\setcounter{equation}{0}

Let $\Omega$ be a bounded  domain in $\mathbb{R}^d$ $(d=2,3)$ with smooth boundary $\partial\Omega$ and a finite time $T$. We consider the incompressible Navier-Stokes equations (INSE)
in primitive variables,
\begin{align}\partial_t\myvec{u}-\nu\Delta\myvec{u}+\myvec{u}\cdot\nabla\myvec{u}+\nabla p&\,=\myvec{f}\quad\text{for $\myvec{x}\in\Omega$, $0<t\le T$,}\label{cont: INSE-momentum}\\
\nabla\cdot\myvec{u}&\,=0\quad\text{for $\myvec{x}\in\Omega$,}\label{cont: INSE-incompressible}\end{align}
subjected to  $\myvec{u}=\myvec{u}^0$ at $t=0$ and the no-slip boundary condition $\myvec{u}=0$ on $\partial\Omega$.
Here $\myvec{u}$ is the velocity of fluid and $p$ is its pressure,
$\nu:=1/\mathrm{Re}$ is the coefficient of the kinematic viscosity  with the
Reynolds number $\mathrm{Re}\ge 1$, and $\myvec{f}$ is a given body force.

We introduce the inner product $\myinner{\myvec{u},\myvec{v}}_{\Omega}:=\int_{\Omega}\myvec{u}\cdot\myvec{v}\zd\myvec{x}$ for the functions $\myvec{u},\myvec{v}$ in $\Omega$, and the associated $L^2$ norm $\mynorm{\myvec{u}}_{\Omega}:=\sqrt{\myinner{\myvec{u},\myvec{u}}_{\Omega}}$. The subscript on the inner product and norm will be dropped when the domain of integration is understood in context.
As an efficient decoupled approach that does not suffer from the splitting error and can achieve full-order accuracy in strong norms, the so-called gauge method
\cite{E-Liu:2003,NochettoPyo:2005} or  well-known consistent splitting technique \cite{GuermondShen:2003,JohnstonLiu:2004,LiuLiuPego:2007,LiuLiuPego:2010,ShenYang:2007} computes the pressure $p$ with a weak form Poisson equation by testing the momentum equation \eqref{cont: INSE-momentum} against gradients. That is, by taking the $L^2$ inner product of \eqref{cont: INSE-momentum} with $\nabla \phi$ and using the incompressible condition
\eqref{cont: INSE-incompressible}, we find that
\begin{align}\label{cont: INSE-pressure Poisson}
	\myinner{\nabla p,\nabla \phi}=\myinner{\myvec{f}+\nu\Delta\myvec{u}-\myvec{u}\cdot\nabla\myvec{u},\nabla \phi}\quad\text{for $\forall\phi\in H^1(\Omega)$.}
\end{align}
This weak form Poisson equation can split the computation of the velocity $\myvec{u}$ and the pressure  $p$ of INSE model \eqref{cont: INSE-momentum}-\eqref{cont: INSE-incompressible} into two consecutive steps. Once the velocity $\myvec{u}$ is known, one can compute the current pressure  $p$ via the Poisson equation \eqref{cont: INSE-pressure Poisson} and update the velocity from the momentum equation \eqref{cont: INSE-momentum} with the no-slip boundary condition by using the previous values of the pressure.  In this way, one can construct fully decoupled time approximations of INSE model \eqref{cont: INSE-momentum}-\eqref{cont: INSE-incompressible}  by the following unconstrained INSE system
\begin{align}
	\myinner{\nabla p,\nabla \phi}&\,=\myinner{\nu\Delta\myvec{u}-\nu\nabla\nabla\cdot\myvec{u}
		-\myvec{g},\nabla \phi}\quad\text{for $\forall\phi\in H^1(\Omega)$,}\label{cont: unconstrained INSE-pressure Poisson}\\
	\partial_t\myvec{u}&\,=\nu\Delta\myvec{u}-\nabla p-\myvec{g}\quad\text{for $\myvec{x}\in\Omega$ with $\myvec{u}=0$ on $\partial\Omega$,}\label{cont: unconstrained INSE-momentum}
\end{align}
subjected to the initial data $\myvec{u}=\myvec{u}^0$ at $t=0$, where the notation $\myvec{g}:=\myvec{u}\cdot\nabla\myvec{u}-\myvec{f}$. Here the grad-div term $\nu\nabla\nabla\cdot\myvec{u}$ in the pressure Poisson equation \eqref{cont: unconstrained INSE-pressure Poisson} is added to improve the stability of numerical approximations \cite{JohnstonLiu:2004,LiuLiuPego:2007,LiuLiuPego:2010}.
As proved in \cite[Theorem 5.1]{LiuLiuPego:2007}, the divergence $w:=\nabla\cdot\myvec{u}$ in the unconstrained INSE system \eqref{cont: unconstrained INSE-pressure Poisson}-\eqref{cont: unconstrained INSE-momentum} is a classical solution of the heat equation $\partial_tw=\nu\Delta w$  with no-flux boundary conditions, $\myvec{n}\cdot\nabla w=0$ on $\partial\Omega$. This property is practically useful in developing stable numerical algorithms since any residues of the numerical divergence would be over controlled especially near the no-slip boundary.
Alternatively, due to the identity $\Delta\myvec{u}=\nabla\nabla\cdot\myvec{u}-\nabla\times\nabla\times\myvec{u}$, the modification in \eqref{cont: unconstrained INSE-pressure Poisson} can be also viewed as the replacement of $\Delta\myvec{u}$ by $-\nabla\times\nabla\times\myvec{u}$ to improve the accuracy of fractional-step splitting procedures \cite{GuermondShen:2003,ShenYang:2007}.

Consider the time mesh $0=t_0<t_1<\cdots<t_N=T$ with the time-step size $\tau =t_j-t_{j-1}$ for $j\ge1$,
and denote the difference quotient $\partial_{\tau}v^j:=(v^j-v^{j-1})/\tau$ for any mesh function $v^j$. Let $\myvec{u}^k$ and $p^k$ be the numerical approximation of $\myvec{U}^k:=\myvec{u}(\myvec{x},t_k)$ and $P^k:=p(\myvec{x},t_k)$, respectively, at the discrete time level $t_k$ for $0\le k\le N$. The classical first-order consistent splitting scheme \cite{LiuLiuPego:2007,LiuLiuPego:2010}
\begin{align}
	\myinner{\nabla p^{n},\nabla \phi}&\,=\myinner{\nu\Delta\myvec{u}^{n}-\nu\nabla\nabla\cdot\myvec{u}^{n}
		-\myvec{g}^{n},\nabla \phi}\quad\text{for  $\forall\phi\in H^1(\Omega)$ and $n\ge0$,}\label{scheme: IERK1-pressure Poisson}\\
	\partial_{\tau}\myvec{u}^{n}&\,=\nu\Delta\myvec{u}^{n}-\nabla p^{n-1}-\myvec{g}^{n-1}
	\quad\text{for $\myvec{x}\in\Omega$ and $n\ge1$,}\label{scheme: IERK1-momentum}
\end{align}
subjected to the boundary condition $\myvec{u}^n=0$ on $\partial\Omega$. The stability result of the semi-discrete first-order consistent splitting scheme for the time dependent Stokes equations with no-slip boundary conditions
(resp., with periodic-nonperiodic boundary conditions in a periodic channel) was established in \cite{GuermondShen:2003}, and the local-in-time stability and error estimates for INSE were proven in a series of work by Liu, Liu and Pego\cite{LiuLiuPego:2007,LiuLiuPego:2010}. By using the second order backward differentiation formula $D_2v^n:=\frac{3}2	\partial_{\tau}v^{n}-\frac{1}2\partial_{\tau}v^{n-1}$ and the second-order explicit extrapolation $\hat{v}^{n}:=2v^{n-1}-v^{n-2}$, a second-order consistent splitting scheme \cite{GuermondShen:2003,JohnstonLiu:2004} was formulated as follows,
\begin{align*}
	\myinner{\nabla p^{n},\nabla \phi}&\,=\myinner{\nu\Delta\myvec{u}^{n}-\nu\nabla\nabla\cdot\myvec{u}^{n}
		-\myvec{g}^{n},\nabla \phi}\quad\text{for  $\forall\phi\in H^1(\Omega)$ and $n\ge0$,}\\
	D_{2}\myvec{u}^{n}&\,=\nu\Delta\myvec{u}^{n}
	-\nabla \hat{p}^{n}-\hat{\myvec{g}}^{n}
	\quad\text{for $\myvec{x}\in\Omega$  with $\myvec{u}^n=0$ on $\partial\Omega$  and $n\ge2$.}
\end{align*}
However, whether the second-order consistent splitting scheme is stable remains open since its inception.
Actually, as pointed out by \cite{HuangShen:2023,HuangShen:2025mcom}, it has been a long standing open question
on how to construct unconditionally stable second- or higher-order decoupled scheme with
a rigorous stability and error analysis.

The situation was essentially changed after the series of works  \cite{HuangShen:2023,HuangShen:2024,HuangShen:2025mcom} by Huang and Shen, in which a class of $\rmk$-th ($2\le\rmk\le 5$) order generalized implicit-explicit backward differentiation formulas (GBDF-$\rmk$) with a free parameter $\beta$ $(\beta\ge1)$ was designed and  applied to develop some  high-order unconditionally stable schemes.
They proved in \cite[Theorem 2]{HuangShen:2024} that the GBDF2 and GBDF3 methods for  $\beta>1$, and the GBDF4 method for  $\beta\ge2$ are stable for linear parabolic problems; while \cite[Theorem 3]{HuangShen:2024} establishes the corresponding convergence for a class of nonlinear parabolic problems by assuming that the nonlinear term is locally Lipschitz continuous.
By choosing three fixed parameters $\beta_{\rmk}=3,6,9$ corresponding to the order indexes $\rmk=2,3,4$, respectively, Huang and Shen \cite{HuangShen:2025mcom} established the unconditional stability and convergence of GBDF-$\rmk$ consistent splitting schemes in the $\ell^{\infty}(H^1)\cap \ell^{2}(H^2)$ norm  for the INSE. Note that, the theoretical analysis in \cite{HuangShen:2024,HuangShen:2025mcom} always takes advantage of the Dahlquist's G-stability theory,
a new convolution-type multiplier and certain  decomposition  of implicit part, cf. \cite[(3.5)]{HuangShen:2024} and \cite[(3.16)]{HuangShen:2025mcom}.

The significant developments in \cite{HuangShen:2023,HuangShen:2024,HuangShen:2025mcom} raise  some new questions: \textsl{why can the  GBDF-$\rmk$ consistent splitting methods achieve the unconditional stability for the INSE, or why choose the free parameters $\beta_{\rmk}=3,6,9$ for the order indexes $\rmk=2,3,4$, respectively? What kind of implicit-explicit linear multistep (IELM) methods can also make the associated consistent splitting algorithms unconditionally stable for the INSE?}
This article answers these questions by presenting a concise, unified theory on the unconditional stability and convergence of general IELM methods for the unconstrained INSE system \eqref{cont: unconstrained INSE-pressure Poisson}-\eqref{cont: unconstrained INSE-momentum}:
\begin{itemize}
	\item With the help of the Stokes pressure estimate from \cite[Theorem 1.2]{LiuLiuPego:2007}, we apply the recent semi-generating function approach \cite{LiaoQuanTangZhou:IMES} and the global discrete energy analysis to establish the unconditional stability of consistent splitting IELM methods with respect to the $\ell^{\infty}(H^1)\cap \ell^{2}(H^2)$ norm if the associated implicit-explicit controllability intensity $\mathfrak{I}_{\mathrm{IE}}^{(\rmk)}>\sqrt{2}/2$.
	\item By choosing proper large parameters, the recently proposed $\beta$-parameterized GBDF-$\rmk$ (for $2\le \rmk\le5$) schemes in \cite{HuangShen:2024,HuangShen:2025mcom} and $\gamma$-parameterized SIELM-$\rmk$ (for $2\le \rmk\le9$) schemes in \cite{LiaoQuanTangZhou:IMES} can fulfill  $\mathfrak{I}_{\mathrm{IE}}^{(\rmk)}>\sqrt{2}/2$ so that they can theoretically achieve the unconditional stability of the associated consistent splitting IELM methods. They present a positive answer to the first open issue proposed by Huang and Shen \cite{HuangShen:2025mcom}.
	\item Extensive experiments are presented to examine the stability and convergence of GBDF-$\rmk$ and  SIELM-$\rmk$ methods for different parameters.  They suggest that our theoretical requirement  $\mathfrak{I}_{\mathrm{IE}}^{(\rmk)}>\sqrt{2}/2$ is sufficient but always not sharp.
\end{itemize}

Note that, our theory would be theoretically concise since the stability and convergence analysis do not involve any Nevanlinna-Odeh type multipliers for high-order implicit-explicit BDF methods \cite{Akrivis:2013,Akrivis:2015,AkrivisChenYuZhou:2021,AkrivisChenYu:2024}  or any delicate decompositions of implicit part for the GBDF-$\rmk$ methods \cite{HuangShen:2024,HuangShen:2025mcom}. 

We consider the following $\rmk$-step consistent splitting IELM method to integrate the INSE from $t_{n-1}$ ($n\ge1$) to the next grid point $t_n$:
\begin{align}	
	\myinner{\nabla p^{n},\nabla \phi}=&\,\myinner{
		\nu\Delta\myvec{u}^{n}-\nu\nabla\nabla\cdot\myvec{u}^{n}-\myvec{g}^{n},\nabla \phi}\quad\text{for $\forall\phi\in H^1(\Omega)$, $n\ge0$,}\label{scheme: general multistep-pressure Poisson}\\
	\sum_{j=0}^{\rmk-1}a_{j}^{(\rmk)}\partial_{\tau}\myvec{u}^{n-j}
	=&\,\nu\sum_{j=0}^{\rmk-1}b_{j}^{(\rmk)}\Delta\myvec{u}^{n-j}
	-\sum_{j=0}^{\rmk-1}c_{j}^{(\rmk)}(\nabla p^{n-j-1}+\myvec{g}^{n-j-1})	\label{scheme: general multistep-momentum}
\end{align}
for $\myvec{x}\in\Omega$ with $\myvec{u}^n=0$ on $\partial\Omega$ and $n\ge\rmk$, where $a_{j}^{(\rmk)}$, $b_{j}^{(\rmk)}$ and $c_{j}^{(\rmk)}$ are the coefficients with $a_{0}^{(\rmk)},b_{0}^{(\rmk)}, c_{0}^{(\rmk)}>0$. One can use a triad $\brab{\vec{a}^{(\rmk)},\vec{b}^{(\rmk)},\vec{c}^{(\rmk)}}$ to represent
the IELM method	 with vectors
$\vec{a}^{(\rmk)}=\brab{a_{0}^{(\rmk)},a_{1}^{(\rmk)},\cdots,a_{\rmk-1}^{(\rmk)}}$,
$\vec{b}^{(\rmk)}=\brab{b_{0}^{(\rmk)},b_{1}^{(\rmk)},\cdots,b_{\rmk-1}^{(\rmk)}}$
and $\vec{c}^{(\rmk)}=\brab{c_{0}^{(\rmk)},c_{1}^{(\rmk)},\cdots,c_{\rmk-1}^{(\rmk)}}$.  Here we assume that the starting solutions $\myvec{u}^{j}$ for $1\le j\le \rmk-1$ are available and accurate enough to meet with the global time accuracy of the IELM method \eqref{scheme: general multistep-momentum}. \lan{The IELM scheme  can be characterized by the following three polynomials
\begin{align*}
	\varrho_a(\zeta):=\sum_{j=0}^{\rmk-1}a_{j}^{(\rmk)}\zeta^{\rmk-1-j},\quad
	\varrho_b(\zeta):=\sum_{j=0}^{\rmk}b_{j}^{(\rmk)}\zeta^{\rmk-j},\quad
	\varrho_c(\zeta):=\sum_{j=0}^{\rmk-1}c_{j}^{(\rmk)}\zeta^{\rmk-1-j}.
\end{align*} 
Following \cite{AkrivisCrouzeixMakridakis:1998MCOM,AkrivisCrouzeixMakridakis:1999NM}, we always let $(\varrho_a,\varrho_b)$ be a strongly $A(0)$-stable implicit scheme, that is, $\varrho_a$ is a Schur polynomial (all
its roots lie strictly inside the unit disk).}

We reformulate the IELM method \eqref{scheme: general multistep-momentum} into a convolution form,
\begin{align}\label{scheme: general imex multistep convolution}	
	\sum_{j=\rmk}^{n}a_{n-j}^{(\rmk)}\partial_{\tau}\myvec{u}^{j}
	=&\,\nu\sum_{j=\rmk}^{n}b_{n-j}^{(\rmk)}\Delta\myvec{u}^{j}
	-\sum_{j=\rmk}^{n}c_{n-j}^{(\rmk)}
	(\nabla p^{j-1}+\myvec{g}^{j-1})-\myvec{u}_{\mathrm{sta}}^{(\rmk,n)}
\end{align}
for $n\ge\rmk$ with the starting term $\myvec{u}_{\mathrm{sta}}^{(\rmk,n)}$ defined by, cf. \cite{JiLiao:2024,LiaoKang:2022,LiaoTangZhou:2021bdf345},
\begin{align}\label{scheme: general imex multistep convolution-initial}	
	\myvec{u}_{\mathrm{sta}}^{(\rmk,n)}:=\sum_{j=1}^{\rmk-1}a_{n-j}^{(\rmk)}\partial_{\tau}\myvec{u}^{j}
	-\nu\sum_{j=1}^{\rmk-1}b_{n-j}^{(\rmk)}\Delta\myvec{u}^{j}
	+\sum_{j=1}^{\rmk-1}c_{n-j}^{(\rmk)}
	(\nabla p^{j-1}+\myvec{g}^{j-1}),
\end{align}
where the values of $a_{j}^{(\rmk)}$, $b_{j}^{(\rmk)}$ and $c_{j}^{(\rmk)}$ are extended to any indexes $j$ but
assume that the coefficients $a_{j}^{(\rmk)}$, $b_{j}^{(\rmk)}$ and $c_{j}^{(\rmk)}$ vanish when $j\ge\rmk$ so that $\myvec{u}_{\mathrm{sta}}^{(\rmk,n)}=0$ for $n\ge2\rmk-1$.

Our framework will use the global discrete energy analysis \cite{LiaoZhang:2021,LiaoKang:2022,LiaoTangZhou:2021bdf345} with the so-called discrete orthogonal convolution (DOC) kernels.  For the finite sequence $\vec{a}^{(\rmk)}$, we define
the DOC kernels $\vec{a}^{(-1,\rmk)}=\big\{a_{0}^{(-1,\rmk)},a_{1}^{(-1,\rmk)},\cdots,a_{\rmk-1}^{(-1,\rmk)},\cdots\big\}$ as follows
\begin{align}\label{def: DOC-Kernels}
	a_{0}^{(-1,\rmk)}:=\frac{1}{a_{0}^{(\rmk)}}
	\quad \mathrm{and} \quad
	a_{j}^{(-1,\rmk)}:=-\frac{1}{a_{0}^{(\rmk)}}
	\sum_{i=1}^{j}a_{j-i}^{(-1,\rmk)}a_{i}^{(\rmk)}\quad \text{for $j\ge1$.}
\end{align}
For $n\ge1$, one has the {discrete orthogonal convolution identity} \cite{LiaoTangZhou:2021bdf345}
\begin{align}\label{eq: orthogonal identity}
	\sum_{\ell=j}^{n}a_{n-\ell}^{(-1,\rmk)}a^{(\rmk)}_{\ell-j}\equiv\delta_{nj}=\sum_{\ell=j}^{n}a_{n-\ell}^{(\rmk)}a^{(-1,\rmk)}_{\ell-j}
	\quad\text{for any $1\leq j\le n$,}
\end{align}
where $\delta_{nj}$ is the Kronecker delta symbol.
Since $a^{(\rmk)}_{\ell}=0$ for $\ell\ge\rmk$, the DOC kernels $a_{j}^{(-1,\rmk)}$ solve the
linear difference equation 	$a_{j}^{(-1,\rmk)}a^{(\rmk)}_{0}+\sum_{\ell=1}^{\rmk-1}a_{j-\ell}^{(-1,\rmk)}a^{(\rmk)}_{\ell}=0$ for $j\ge\rmk$ so that it has the solution $a_{j}^{(-1,\rmk)}=\sum_{\ell=1}^{\rmk-1}\ck_{\ell}^*\lambda_{\ell}^j$, where $\lambda_{\ell}$ $(1\le \ell\le\rmk-1)$ are zero points of the Schur polynomial   $\varrho_{a}(\zeta)$ and $\ck_{\ell}^*$ are constants determined by the initial values $a_{\ell}^{(-1,\rmk)}$ for $0\le\ell\le \rmk-1$. Then there exist positive constants $\ck_{a}$ and $\rho_{\rmk}<1$ such that
\begin{align}\label{ieq: DOC-Kernels decaying}
	\absb{a_{j}^{(-1,\rmk)}}\le \ck_{a}\rho_{\rmk}^j\quad \text{for $j\ge1$,}
\end{align}
that is, the DOC kernels $a_{j}^{(-1,\rmk)}$ decay exponentially.
By exchanging the summation order, one gets
\begin{align*}
	\sum_{j=\rmk}^{n}a_{n-j}^{(-1,\rmk)}
	\sum_{\ell=\rmk}^{j}a_{j-\ell}^{(\rmk)}\partial_{\tau}\myvec{u}^{\ell}
	=&\,\sum_{\ell=\rmk}^{n}\partial_{\tau}\myvec{u}^{\ell}\sum_{j=\ell}^{n}a_{n-j}^{(-1,\rmk)}a_{j-\ell}^{(\rmk)}
	=\partial_{\tau}\myvec{u}^n\quad\text{for $n\ge\rmk$.}
\end{align*}
Multiplying the equation \eqref{scheme: general imex multistep convolution}	 with
the DOC kernels $a_{m-n}^{(-1,\rmk)}$, summing over $n$ from $n=\rmk$ to $m$ and replacing $m$ by $n$, we get
\begin{align}\label{Dis: DOC action multistep formula Dk}
	&\,\sum_{j=\rmk}^{n}a_{n-j}^{(-1,\rmk)}
	\sum_{\ell=\rmk}^ja_{j-\ell}^{(\rmk)}\partial_{\tau}\myvec{u}^{\ell}
	=\nu\sum_{j=\rmk}^{n}a_{n-j}^{(-1,\rmk)}\sum_{\ell=\rmk}^{j}b_{j-\ell}^{(\rmk)}\Delta\myvec{u}^{\ell}\\
	&\,\hspace{1.8cm}
	-\sum_{j=\rmk}^{n}a_{n-j}^{(-1,\rmk)}\sum_{\ell=\rmk}^{j}c_{j-\ell}^{(\rmk)}(\nabla p^{\ell-1}+\myvec{g}^{\ell-1})
	-\sum_{j=\rmk}^{n}a_{n-j}^{(-1,\rmk)}\myvec{u}_{\mathrm{sta}}^{(\rmk,j)}\notag
\end{align}
for $\myvec{x}\in\Omega$ and $n\ge\rmk$. By exchanging the summation order, one can apply discrete orthogonal convolution identity \eqref{eq: orthogonal identity} to find
\begin{align}\label{scheme: general multistep-momentum-differential}
	\partial_{\tau}\myvec{u}^{n}
	=&\,\nu\sum_{\ell=\rmk}^{n}\hat{b}_{n-\ell}^{(\rmk)}\Delta\myvec{u}^{\ell}
	-\sum_{\ell=\rmk}^{n}\hat{c}_{n-\ell}^{(\rmk)}(\nabla p^{\ell-1}+\myvec{g}^{\ell-1})
	-\sum_{\ell=\rmk}^{n}a_{n-\ell}^{(-1,\rmk)}\myvec{u}_{\mathrm{sta}}^{(\rmk,\ell)}
\end{align}
for $\myvec{x}\in\Omega$ and $n\ge\rmk$, where the composited kernels $\hat{b}_{n-\ell}^{(\rmk)}$ and $\hat{c}_{n-\ell}^{(\rmk)}$ are defined by
\begin{align}\label{scheme: multistep composited kernels}
	\hat{b}_{j}^{(\rmk)}:=\sum_{i=0}^{j}a_{j-i}^{(-1,\rmk)}b_{i}^{(\rmk)}\quad\text{and}\quad
	\hat{c}_{j}^{(\rmk)}:=\sum_{i=0}^{j}a_{j-i}^{(-1,\rmk)}c_{i}^{(\rmk)}\quad\text{for $j\ge0$}.
\end{align}

For the coefficients $a_{n-j}^{(\rmk)}$ and the associated DOC kernels $a_{n-j}^{(-1,\rmk)}$ defined by \eqref{def: DOC-Kernels}, we introduce the following $n\times n$ lower triangular Toeplitz matrices
\begin{align*}
	&A_{L,\rmk}:=\begin{pmatrix}
		a_0^{(\rmk)} & && &\\
		\vdots                     & \ddots & &    &\\
		a_{\rmk-1}^{(\rmk)}&\cdots &a_0^{(\rmk)}&&\\
		&\ddots&\cdots&\ddots&\\
		&&a_{\rmk-1}^{(\rmk)}&\cdots&a_0^{(\rmk)}
	\end{pmatrix},\quad
	A_{L,\rmk}^{(-1)}:=
	\begin{pmatrix}
		a_0^{(-1,\rmk)} & &&\\
		a_1^{(-1,\rmk)} &a_0^{(-1,\rmk)} && \\
		a_2^{(-1,\rmk)} &a_1^{(-1,\rmk)} &a_0^{(-1,\rmk)}& \\
		\vdots            &    \ddots          & \ddots  &    \\
		a_{n-1}^{(-1,\rmk)}&\cdots&a_1^{(-1,\rmk)} &a_0^{(-1,\rmk)}
	\end{pmatrix}.
\end{align*}
The discrete orthogonal convolution identity \eqref{eq: orthogonal identity} says that  $A_{L,\rmk}^{(-1)}=A_{L,\rmk}^{-1}.$ In a similar way, one can write out the lower triangular Toeplitz matrices $B_{L,\rmk}$ and $C_{L,\rmk}$ from the discrete coefficients $b_{n-j}^{(\rmk)}$ and $c_{n-j}^{(\rmk)}$; while $B_{L,\rmk}^{(-1)}=B_{L,\rmk}^{-1}$ and $C_{L,\rmk}^{(-1)}=C_{L,\rmk}^{-1}$ are the lower triangular Toeplitz matrices for the corresponding DOC kernels $b_{n-j}^{(-1,\rmk)}$ and $c_{n-j}^{(-1,\rmk)}$ defined in similar to \eqref{def: DOC-Kernels} from $b_{n-j}^{(\rmk)}$ and $c_{n-j}^{(\rmk)}$, respectively. Moreover,
for the composited kernels $\hat{b}_{n-\ell}^{(\rmk)}$ and $\hat{c}_{n-\ell}^{(\rmk)}$ defined in \eqref{scheme: multistep composited kernels}, it is easy to  know that the corresponding lower triangular Toeplitz matrices are
\begin{align}  \label{matrix: hatB_L hatC_L}
	\widehat{B}_{L,\rmk}:=&\,A_{L,\rmk}^{(-1)}B_{L,\rmk}=A_{L,\rmk}^{-1}B_{L,\rmk},\quad
	\widehat{C}_{L,\rmk}:=A_{L,\rmk}^{(-1)}C_{L,\rmk}=A_{L,\rmk}^{-1}C_{L,\rmk}.
\end{align}

In the energy analysis of the  convolution form \eqref{scheme: general multistep-momentum-differential},
the treatment of implicit part requires to determine  the minimum eigenvalue $\lambda_{\mathrm{I}}^{(\rmk)}$ of the Toeplitz matrix $\mathcal{S}(\widehat{B}_{L,\rmk})$, where  $\mathcal{S}(D):=(D+D^T)/2$ for any matrix $D$,
while certain spectral norm bounds $\sigma_{\mathrm{E}}^{(\rmk)}$ and $\sigma_{\mathrm{F}}^{(\rmk)}$ of the  lower triangular Toeplitz matrices
$\widehat{C}_{L,\rmk}$ and $A_{L,\rmk}^{-1}$ should be evaluated in handling the explicit and starting terms. These issues are addressed in Section \ref{sec: semi-generating function method} by the semi-generating function method.

With the help of the Leray-Helmholtz projection operator, the stability analysis of consistent splitting multistep methods \eqref{scheme: general multistep-pressure Poisson}-\eqref{scheme: general multistep-momentum} will be performed in Section \ref{sec: Stability of general IELM} by a complete mathematical induction to the $\ell^{\infty}(H^1)\cap \ell^{2}(H^2)$ norm boundedness of the velocity error.
It is shown that the IELM methods \eqref{scheme: general multistep-pressure Poisson}-\eqref{scheme: general multistep-momentum} are unconditionally stable if the controllability intensity  $\mathfrak{I}_{\mathrm{IE}}^{(\rmk)}:=\lambda_{\mathrm{I}}^{(\rmk)}/\sigma_{\mathrm{E}}^{(\rmk)}>\frac{\sqrt{2}}2$.
Section \ref{sec: existing IELM} witnesses that the stability condition can be fulfilled by two classes of IELM methods, including  GBDF schemes up to fifth order accuracy and SIELM schemes up to ninth order accuracy.
Numerical experiments are included in Section \ref{sec: numerical example} to support our theory.

\section{Semi-generating function method and technical lemmas}\label{sec: semi-generating function method}
\setcounter{equation}{0}

At first, we recall the semi-generating function method to handle the three groups of
discrete coefficients via three complex polynomials on the unit circle.

\begin{lemma}\cite[Lemma 2.1]{LiaoQuanTangZhou:IMES}\label{lem: Toeplitz-Caratheodory}
	For a real sequence $\{a_{0},a_{1},\cdots,a_{n},\cdots\}$, 	we define the semi-generating function $a(\theta):=\sum_{k=0}^{\infty}a_k e^{\imath k\theta }{\in L^2([0,2\pi))}$ with the complex unit $\imath=\sqrt{-1}$.
	For any index $n\ge1$ (while $n\rightarrow\infty$ as the time-step size $\tau\rightarrow0$), consider the following real quadratic form
	\begin{align*}
		Q_n:=\sum_{k=1}^nw_k\sum_{j=1}^ka_{k-j}w_j\quad\text{for any sequence $\{w_{1}, w_{2}, \cdots, w_{n}\}$,}
	\end{align*}
	corresponding to the real symmetric Toeplitz matrix $\mathcal{S}(P_{L, n})=(P_{L, n}+P_{L, n}^T)/2$ with
	the following associated  lower triangular Toeplitz matrix
	\begin{align}\label{lower triangular Toeplitz matrix PL}
		P_{L, n}:=\left(
		\begin{array}{ccccc}
			a_0 & && \\
			a_1 &a_0 & &\\
			\vdots            &    \ddots          & \ddots & \\
			a_{n-1}&\cdots &a_1&a_0\\
		\end{array}
		\right)_{n\times n}.
	\end{align}
	\begin{itemize}[itemindent=-0.5cm]
		\item[(i)] Then  $Q_n$ is positive definite if and only if $\Re\kbra{a(\theta)}>0$ for $\theta\in[0,2\pi)$;
		\item[(ii)]  and the eigenvalues $\lambda_j(Q_n)$ of $Q_n$ can be bounded by
		\begin{align*}
			\min_{\theta\in[0,2\pi)}\Re\kbrab{a(\theta)}\le\lambda_j(Q_n)\le \max_{\theta\in[0,2\pi)}\Re\kbrab{a(\theta)}\quad\text{for any $n\ge j+1\ge1$.}
		\end{align*}
	\end{itemize}
\end{lemma}

\begin{lemma}\cite[Lemma 2.2]{LiaoQuanTangZhou:IMES}\label{lem: composited generating function}
	Let the functions $a(\theta):=\sum_{j=0}^{\infty}a_je^{\imath j\theta}{\in L^2([0,2\pi))}$ and $b(\theta):=\sum_{j=0}^{\infty}b_je^{\imath j\theta}{\in L^2([0,2\pi))}$ be the semi-generating functions for real sequences $\{a_{0},a_{1},\cdots,a_{k},\cdots\}$  and $\{b_{0},b_{1},\cdots,b_{k},\cdots\}$, respectively.
	\begin{itemize}[itemindent=-0.5cm]
		\item[(i)] 	For the composited sequence $\{\hat{b}_{0},\hat{b}_{1},\cdots,\hat{b}_{k},\cdots\}$ defined by $\hat{b}_{j}:=\sum_{k=0}^{j}a_{j-k}b_{k},$
		the semi-generating function $\hat{b}(\theta)=\sum_{j=0}^{\infty}\hat{b}_je^{\imath j\theta}$ satisfies $\hat{b}(\theta)=a(\theta)b(\theta)$.
		\item[(ii)] Assume that the real sequence $\{\xi_{0},\xi_{1},\cdots,\xi_{k},\cdots\}$ is the DOC kernels of  $\{a_{0},a_{1},\cdots,a_{k},\cdots\}$, defined by  $\xi_0:=\frac1{a_0}$ and $\xi_j:=-\frac1{a_0}\sum_{k=1}^{j}\xi_{j-k}a_{k}$ for $j\ge1$.
		Then
		the associated semi-generating function $\xi(\theta)=\sum_{j=0}^{\infty}\xi_je^{\imath j\theta}$ satisfies $\xi(\theta)=1/a(\theta)$.		
	\end{itemize}
\end{lemma}

\begin{lemma}\cite[Lemma 2.3]{LiaoQuanTangZhou:IMES}\label{lemma: spectral norm bound}
	For the lower triangular Toeplitz matrix $P_{L,n}$ in \eqref{lower triangular Toeplitz matrix PL}
	and the associated semi-generating function $a(\theta)=\sum_{k=0}^{\infty}a_k e^{\imath k\theta }{\in L^2([0,2\pi))}$,
	the spectral norm of $P_{L,n}$ is not larger than $\max_{\theta\in[0,2\pi)}\abs{a(\theta)}$ for any $n\ge1$.
\end{lemma}

For the discrete coefficients $a_{j}^{(\rmk)}$, $b_{j}^{(\rmk)}$ and $c_{j}^{(\rmk)}$  of the $\rmk$-step multistep method \eqref{scheme: general multistep-momentum},
Lemma \ref{lem: Toeplitz-Caratheodory} defines the associated semi-generating functions
\begin{align}\label{matrix: A_L B_L C_L generating function}
	a^{(\rmk)}(\theta):=\sum_{j=0}^{\rmk-1}a_{j}^{(\rmk)}e^{\imath j\theta},\quad
	b^{(\rmk)}(\theta):=\sum_{j=0}^{\rmk-1}b_{j}^{(\rmk)}e^{\imath j\theta}\quad\text{and}\quad
	c^{(\rmk)}(\theta):=\sum_{j=0}^{\rmk-1}c_{j}^{(\rmk)}e^{\imath j\theta}.
\end{align}
Lemma \ref{lem: composited generating function} (ii) gives the semi-generating function
for the DOC kernels $a_{n-j}^{(-1,\rmk)}$,
\begin{align}\label{matrix: A_L DOC generating function}
	a^{(-1,\rmk)}(\theta):=\sum_{j=0}^{\infty}a_{j}^{(-1,\rmk)}e^{\imath j\theta}
	=\frac{1}{a^{(\rmk)}(\theta)}.
\end{align}
For the composited discrete kernels $\hat{b}_{n-\ell}^{(\rmk)}$ and $\hat{c}_{n-\ell}^{(\rmk)}$ defined in \eqref{scheme: multistep composited kernels},
Lemma \ref{lem: composited generating function} (i)  gives the associated semi-generating functions
\begin{align}\label{matrix: hatB_L hatC_L generating function}
	\hat{b}^{(\rmk)}(\theta):=&\,\sum_{j=0}^{\infty}\hat{b}_{j}^{(\rmk)}e^{\imath j\theta}
	=a^{(-1,\rmk)}(\theta)b^{(\rmk)}(\theta)=\frac{b^{(\rmk)}(\theta)}{a^{(\rmk)}(\theta)},\quad
	\hat{c}^{(\rmk)}(\theta):=\frac{c^{(\rmk)}(\theta)}{a^{(\rmk)}(\theta)}.
\end{align}

The following result builds some close relationships between the semi-generating functions of lower triangular Toeplitz matrices and our discrete energy techniques for IELM methods.

\begin{lemma}\cite[Lemma 2.4]{LiaoQuanTangZhou:IMES}\label{lemma: bound quadratic form}	
\lan{	Assume that the implicit part of the IELM method \eqref{scheme: general multistep-momentum}	 is strongly $A(0)$-stable.}   For the semi-generating functions $a^{(\rmk)}(\theta)$, $b^{(\rmk)}(\theta)$ and $c^{(\rmk)}(\theta)$ defined in \eqref{matrix: A_L B_L C_L generating function}, assume that there exist finite constants $\sigma_{\mathrm{F}}^{(\rmk)}>0$, $\sigma_{\mathrm{E}}^{(\rmk)}>0$ and $\lambda_{\mathrm{I}}^{(\rmk)}>0$  such that
	\begin{align}\label{def: lambda sigma}
		\sigma_{\mathrm{F}}^{(\rmk)}=\max\limits_{\theta\in[0,2\pi)}\abs{\frac{1}{a^{(\rmk)}(\theta)}},\;\;
		\sigma_{\mathrm{E}}^{(\rmk)}=\max\limits_{\theta\in[0,2\pi)}\abs{\frac{c^{(\rmk)}(\theta)}{a^{(\rmk)}(\theta)}},	\;\;
		\lambda_{\mathrm{I}}^{(\rmk)}=\min_{\theta\in[0,2\pi)}
		\Re\kbra{\frac{b^{(\rmk)}(\theta)}{a^{(\rmk)}(\theta)}}.
	\end{align}
	Then, for $n\ge1$, the spectral
	norms of the lower triangular Toeplitz matrices  $A_{L,\rmk}^{-1}$ and $A_{L,\rmk}^{-1}C_{L,\rmk}$ are bounded by the positive constants $\sigma_{\mathrm{F}}^{(\rmk)}$ and $\sigma_{\mathrm{E}}^{(\rmk)}$, respectively, and all eigenvalues of the symmetric matrix $\mathcal{S}(A_{L,\rmk}^{-1}B_{L,\rmk})$ are not less than $\lambda_{\mathrm{I}}^{(\rmk)}$.
	For any sequences $\{v^{i},u^i: i \geq 1\}$, it holds that
	\begin{align*}
		(i)\;\;&\,\sum_{i=1}^{n} \sum_{j=1}^{i}\hat{b}_{i-j}^{(\rmk)}v^{j} v^{i} \geq \lambda_{\mathrm{I}}^{(\rmk)}\sum_{i=1}^{n}\absb{v^{i}}^2\quad\text{for $n\ge1$,}	\\
		(ii)\;\;&\,\sum_{i=1}^{n} \sum_{j=1}^{i}a_{i-j}^{(-1,\rmk)}v^{j} u^{i} \leq
		\sigma_{\mathrm{F}}^{(\rmk)}\sqrt{\sum_{i=1}^n\absb{v^{i}}^2}\sqrt{\sum_{i=1}^n\absb{u^{i}}^2} \quad\text{for $n\ge1$,}\\
		(iii)\;\;&\,\sum_{i=1}^{n} \sum_{j=1}^{i}\hat{c}_{i-j}^{(\rmk)}v^{j}u^{i} \leq
		\sigma_{\mathrm{E}}^{(\rmk)}\sqrt{\sum_{i=1}^n\absb{v^{i}}^2}\sqrt{\sum_{i=1}^n\absb{u^{i}}^2} \quad\text{for $n\ge1$.}	
	\end{align*}
\end{lemma}


As shown later, the boundedness of $\sigma_{\mathrm{F}}^{(\rmk)}$,   $\sigma_{\mathrm{E}}^{(\rmk)}$ and $\lambda_{\mathrm{I}}^{(\rmk)}$ defined in \eqref{def: lambda sigma} will be essential in the discrete energy analysis for the stability of IELM methods \eqref{scheme: general multistep-momentum}.
From the perspective of the transformed IELM scheme \eqref{scheme: general multistep-momentum-differential}  and the associated discrete energy  analysis with respect to the $\ell^{\infty}(H^1)\cap \ell^{2}(H^2)$ norm, cf. Section \ref{sec: Stability of general IELM},   the spectral norm bound $\sigma_{\mathrm{F}}^{(\rmk)}$ in \eqref{def: lambda sigma} represents the ability  of the DOC kernels  $a_{i-j}^{(-1,\rmk)}$ to control the numerical growth of the starting term $\myvec{u}_{\mathrm{sta}}$ or truncation errors. The smaller the value of $\sigma_{\mathrm{F}}^{(\rmk)}$, the less the DOC kernels amplify the perturbation error. The spectral norm bound $\sigma_{\mathrm{E}}^{(\rmk)}$ in \eqref{def: lambda sigma} represents the ability of the composited kernels  $\hat{c}_{i-j}^{(\rmk)}$ to restrain the possible instability due to the explicit approximation of  $\myvec{g}$ and $\nabla p$. The smaller the value of $\sigma_{\mathrm{E}}^{(\rmk)}$, the better the composited discrete kernels  $\hat{c}_{i-j}^{(\rmk)}$ control the numerical instability. Meanwhile, the minimum eigenvalue $\lambda_{\mathrm{I}}^{(\rmk)}$ in \eqref{def: lambda sigma} represents the ability of the composited discrete kernels  $\hat{b}_{i-j}^{(\rmk)}$ in the implicit part to maintain the dissipation property of $\nu\Delta\myvec{u}$.
The larger the value of $\lambda_{\mathrm{I}}^{(\rmk)}$, the better the composited discrete kernels  $\hat{b}_{i-j}^{(\rmk)}$ inherit the dissipativity of $\nu\Delta\myvec{u}$.

In this sense, we introduce the implicit-explicit controllability intensity $\mathfrak{I}_{\mathrm{IE}}^{(\rmk)}$ in \cite{LiaoQuanTangZhou:IMES}, defined by the ratio of the minimum eigenvalue (dissipation preserving factor) $\lambda_{\mathrm{I}}^{(\rmk)}$ from the implicit part over the spectral norm bound (nonlinear amplification factor) $\sigma_{\mathrm{E}}^{(\rmk)}$ from the explicit part,
\begin{align}\label{def: stability intensity}
	\mathfrak{I}_{\mathrm{IE}}^{(\rmk)}:=\frac{\lambda_{\mathrm{I}}^{(\rmk)}}{\sigma_{\mathrm{E}}^{(\rmk)}}
	=\frac{\min\limits_{\theta\in[0,2\pi)}
		\Re\kbra{\frac{b^{(\rmk)}(\theta)}{a^{(\rmk)}(\theta)}}}{\max\limits_{\theta\in[0,2\pi)}\abs{\frac{c^{(\rmk)}(\theta)}{a^{(\rmk)}(\theta)}}}\,,
\end{align}
where the deduced formula follows from the definitions in \eqref{def: lambda sigma}. It would represent the degree of controllability of the implicit part over the explicit part of a given IELM method.
\begin{lemma}\cite[Lemma 2.6 and Corollary 3.4]{LiaoQuanTangZhou:IMES}
	\label{lem: IELM stability intensity upper bound}
	Assume that the $\rmk$-th step IELM method \eqref{scheme: general multistep-momentum} with the discrete coefficients $a_{j}^{(\rmk)}$, $b_{j}^{(\rmk)}$ and $c_{j}^{(\rmk)}$ are consistent. It holds that
	$	\sigma_{\mathrm{F}}^{(\rmk)}\ge1$,  $\sigma_{\mathrm{E}}^{(\rmk)}\ge1$ and $\lambda_{\mathrm{I}}^{(\rmk)}\le 1$ such that $\mathfrak{I}_{\mathrm{IE}}^{(\rmk)}\le1$.
	While, the implicit-explicit Euler scheme achieves the optimal values,  $\sigma_{\mathrm{F}}^{(1)}=1$, $\sigma_{\mathrm{E}}^{(1)}=1$, $\lambda_{\mathrm{I}}^{(1)}=1$ and $\mathfrak{I}_{\mathrm{IE}}^{(1)}=1$.
\end{lemma}

To handle the nonlinear convection term $\myvec{u}\cdot\nabla\myvec{u}$, we need the following lemma, see \cite[(4.9)-(4.14)]{LiuLiuPego:2007}, which can be derived from
the Ladyzhenskaya's inequalities, the embedding from $H^1(\Omega)$ into $L^4(\Omega)$ and $L^6(\Omega)$,
and the elliptic regularity estimate.
\begin{lemma}\label{lem: nonlinear convective bound}	
	For $\myvec{u},\myvec{v}\in H^2(\Omega,\mathbb R^d)\cap H_0^1(\Omega,\mathbb R^d)$, there exists a constant $\ck_l>0$ such that
	\begin{align*}
		\mynormb{\myvec{u}\cdot\nabla\myvec{v}}^2\le &\,
		\ck_l\mynormb{\myvec{u}}\mynormb{\nabla\myvec{u}}
		\mynormb{\nabla\myvec{v}}\mynormb{\Delta\myvec{v}}\quad\text{for $d=2$,}\\
		\mynormb{\myvec{u}\cdot\nabla\myvec{v}}^2\le&\,
		\ck_l\mynormb{\nabla\myvec{u}}^2
		\mynormb{\nabla\myvec{v}}\mynormb{\Delta\myvec{v}}\quad\text{for $d=2,3$.}		
	\end{align*}
\end{lemma}

To bound the pressure gradient, it is to introduce
the Stokes pressure $p_{\mathrm{s}}=p_{\mathrm{s}}(\myvec{u})$ for any $\myvec{u}\in H^2(\Omega,\mathbb R^d)$, defined from
$\nabla p_{\mathrm{s}}(\myvec{u}):=(\Delta \mathcal{P}-\mathcal{P}\Delta)\myvec{u}.$
Here $\mathcal{P}$ is the Leray-Helmholtz projection operator onto divergence-free fields with zero normal component, providing the Helmholtz decomposition $\myvec{u}=\mathcal{P}\myvec{u}+\nabla \phi$, where
$$\myinner{\mathcal{P}\myvec{u},\nabla \varphi}=\myinner{\myvec{u}-\nabla \phi,\nabla\varphi}=0\quad\text{$\forall \varphi\in H^1(\Omega)$.}$$
As presented in \cite[(2.1)-(2.2)]{LiuLiuPego:2007}, $p_{\mathrm{s}}(\myvec{u})$ is determined as the mean-zero solution of
\begin{align}   \label{conti: Stokes pressure}
	\myinner{\nabla p_{\mathrm{s}}(\myvec{u}),\nabla \varphi}=\myinner{\Delta \myvec{u}-\nabla\nabla\cdot\myvec{u},\nabla \varphi}\quad\text{$\forall \varphi\in H^1(\Omega)$.}
\end{align}
Also, we have the following result.
\begin{lemma}\cite[Theorem 1.2]{LiuLiuPego:2007}\label{lem: Stokes pressure bound}
	Let $\Omega\subset\mathbb{R}^d$ $(d\ge2)$ be a connected bounded domain with $C^3$ boundary $\partial\Omega$.
	Then for any constant $\varepsilon>0$, there exists a constant $\ck_{\varepsilon}>0$ such that for all vector fields $\myvec{u}\in H^2(\Omega,\mathbb R^d)\cap H^1_0(\Omega,\mathbb R^d)$,
	$$\mynorm{(\Delta \mathcal{P}-\mathcal{P}\Delta)\myvec{u}}^2
	\le(\tfrac12+\varepsilon)\mynorm{\Delta\myvec{u}}^2+\ck_{\varepsilon}\mynorm{\nabla\myvec{u}}^2.$$
\end{lemma}

The constant $\tfrac12+\varepsilon$ before the norm $\mynorm{\Delta\myvec{u}}$ would be essential to  maintain the unconditional stability of any numerical methods with certain explicit approximation of pressure gradient $\nabla p$ because the involved norm $\mynorm{\Delta\myvec{u}}$ should be properly controlled by using the dissipation term $\nu\Delta\myvec{u}$.


\section{Stability of general IELM methods}\label{sec: Stability of general IELM}
\setcounter{equation}{0}

\begin{theorem}\label{thm: NS multistep stability}
	Assume that there exists a finite time $T_*>0$ such that the solution pair $(\myvec{u},p)$ of the  INSE model \eqref{cont: INSE-momentum}-\eqref{cont: INSE-incompressible} on the spatial domain $\Omega$ is regular and  fulfills $\myvec{u}\in { C([0,T_*];H_0^1(\Omega)\cap H^2(\Omega))}$,
	$\partial_{t}\myvec{u}\in { L^2([0,T_*];H^2(\Omega))}$, $\partial_{tt}\myvec{u}\in { L^2([0,T_*];L^2(\Omega))}$ and $\partial_{t}p\in{  L^2([0,T_*];H^1(\Omega))}$.
	\lan{Assume further that  the  IELM  method \eqref{scheme: general multistep-momentum} satisfies the assumptions of Lemma \ref{lemma: bound quadratic form}}.	If the implicit-explicit controllability intensity $\mathfrak{I}_{\mathrm{IE}}^{(\rmk)}>\frac{\sqrt{2}}{2}$,
	and the time-step size $\tau$ is sufficiently small, the consistent splitting multistep methods \eqref{scheme: general multistep-pressure Poisson}-\eqref{scheme: general multistep-momentum} for the  INSE model \eqref{cont: INSE-momentum}-\eqref{cont: INSE-incompressible} are unconditionally stable, provided the starting solutions $(\myvec{u}^{j},p^{j})$ for $1\le j\le \rmk-1$ are sufficiently accurate.
\end{theorem}
\begin{proof}
	Let the exact solutions $(\myvec{U}^{n},P^{n})$ solve the time-discrete system \eqref{scheme: general multistep-pressure Poisson}-\eqref{scheme: general multistep-momentum} with the truncation error $\myvec{d}^{n}$ for $1\le n\le N$.
	According to the regularity assumption of solution, there exists a constant $\ck_{\myvec{u}}>0$ such that
	\begin{align*}
		\mynormb{\nabla\myvec{U}^{n}}^2+\mynormb{\Delta\myvec{U}^{n}}^2
		+\sum_{k=1}^n\tau\mynormb{\Delta\myvec{U}^{k}}^2+\sum_{k=1}^n\tau\mynormb{\nabla P^{k}}^2\le \ck_{\myvec{u}}^2
		\quad \text{for $1\le n\le N$.}
	\end{align*}	
	Let the solution errors  $\tilde{\myvec{u}}^{n}:=\myvec{U}^{n}-\myvec{u}^{n}$ and $\tilde{p}^{n}:=P^{n}-p^{n}$.
	Also, define the nonlinear error $\tilde{\myvec{g}}^{n}$ by
	\begin{align}\label{error: nonlinear convection}
		\tilde{\myvec{g}}^{n}:=\myvec{U}^{n}\cdot\nabla\myvec{U}^{n}-\myvec{u}^{n}\cdot\nabla\myvec{u}^{n}
		=\tilde{\myvec{u}}^{n}\cdot\nabla\myvec{u}^{n}+\myvec{U}^{n}\cdot\nabla\tilde{\myvec{u}}^{n}.
	\end{align}
	Then it is easy to derive the error equations with the no-slip  boundary condition $\tilde{\myvec{u}}^{n}=0$ on $\partial\Omega$,
	\begin{align}	
		\myinner{\nabla \tilde{p}^{n},\nabla \phi}=&\,\myinner{
			\nu\Delta\tilde{\myvec{u}}^{n}-\nu\nabla\nabla\cdot\tilde{\myvec{u}}^{n}-\tilde{\myvec{g}}^{n},\nabla \phi}
		\quad \text{for $\forall\phi\in H^1(\Omega)$ and $n\ge0$,}
		\label{scheme: perturbed multistep-pressure Poisson}\\
		\sum_{j=\rmk}^na_{n-j}^{(\rmk)}\partial_{\tau}\tilde{\myvec{u}}^{j}
		=&\,\nu\sum_{j=\rmk}^{n}b_{n-j}^{(\rmk)}\Delta \tilde{\myvec{u}}^{j}
		-\sum_{j=\rmk}^{n}c_{n-j}^{(\rmk)}(\nabla \tilde{p}^{j-1}+\tilde{\myvec{g}}^{j-1})-\tilde{\myvec{u}}_{\mathrm{sta}}^{(\rmk,n)}+\myvec{d}^{n}
		\label{scheme: perturbed multistep-momentum}
	\end{align}
	for $n\ge\rmk$ with the starting error term $\tilde{\myvec{u}}_{\mathrm{sta}}^{(\rmk,n)}$ defined by
	\begin{align}\label{scheme: general imex multistep convolution-initial error}	
		\tilde{\myvec{u}}_{\mathrm{sta}}^{(\rmk,n)}:=
		\sum_{j=1}^{\rmk-1}a_{n-j}^{(\rmk)}\partial_{\tau}\tilde{\myvec{u}}^{j}
		-\nu\sum_{j=1}^{\rmk-1}b_{n-j}^{(\rmk)}\Delta\tilde{\myvec{u}}^{j}
		+\sum_{j=1}^{\rmk-1}c_{n-j}^{(\rmk)}
		(\nabla \tilde{p}^{j-1}+\tilde{\myvec{g}}^{j-1}),
	\end{align}
	which vanishes for $n\ge 2\rmk-1$.
	Following the derivation of \eqref{scheme: general multistep-momentum-differential}, it is not difficult to derive the equivalent convolution form of the error equation \eqref{scheme: perturbed multistep-momentum},
	\begin{align} \label{scheme: perturbed multistep-momentum-differential}
		\partial_{\tau}\tilde{\myvec{u}}^{i}
		=&\,\nu\sum_{\ell=\rmk}^{i}\hat{b}_{i-\ell}^{(\rmk)}\Delta\tilde{\myvec{u}}^{\ell}
		-\sum_{\ell=\rmk}^{i}\hat{c}_{i-\ell}^{(\rmk)}(\nabla \tilde{p}^{\ell-1}+\tilde{\myvec{g}}^{\ell-1})
		+\sum_{\ell=\rmk}^{i}a_{i-\ell}^{(-1,\rmk)}\myvec{R}^{\ell}\quad\text{for $i\ge\rmk$},
	\end{align}
	where $\myvec{R}^{\ell}=\myvec{d}^{\ell}-\tilde{\myvec{u}}_{\mathrm{sta}}^{(\rmk,\ell)}$ for $\ell\ge\rmk$.
	The time truncation error $\myvec{d}^{n}=O(\tau)$ according to the solution regularity. Assume further that the starting solutions $(\myvec{u}^{j},p^{j})$ for $1\le j\le \rmk-1$ are accurate enough and there exists a constant $\tilde{\ck}_{\myvec{u}}>0$ such that
	\begin{align}\label{eq: initial and trunction error estimate}	
		&\,\mynormb{\nabla\tilde{\myvec{u}}^{\rmk-1}}^2+\tau\mynormb{\Delta \tilde{\myvec{u}}^{\rmk-1}}^2
		+\tau \mynormb{\nabla \tilde{p}^{\rmk-1}}^2+\sum_{j=\rmk}^{n}\tau\mynormb{\myvec{R}^{j}}^2\le \tilde{\ck}_{\myvec{u}}^2\tau^2.
	\end{align}

		We consider the mathematical induction for the  error bound in the $\ell^{\infty}(H^1)\cap \ell^{2}(H^2)$ norm
		\begin{align}\label{error: NS multistep H2 norm bound}
			\mynormb{\nabla\tilde{\myvec{u}}^{\ell}}^2+\sum_{j=\rmk}^\ell\tau\mynormb{\Delta\tilde{\myvec{u}}^{j}}^2\le 1
			\quad\text{for $\rmk-1\le \ell\le N$}.
		\end{align}
		It holds for $\ell=\rmk-1$ if the time-step size $\tau\le 1/\tilde{\ck}_{\myvec{u}}$. It needs to verify the case $\ell=m$ $(m\ge \rmk)$ from the  induction hypothesis
		\begin{align}\label{error: NS multistep H2 norm induction hypothesis} 		
			\mynormb{\nabla\tilde{\myvec{u}}^{\ell}}^2+\sum_{j=\rmk}^\ell\tau\mynormb{\Delta\tilde{\myvec{u}}^{j}}^2\le 1
			\quad\text{for $\rmk-1\le \ell\le m-1$}.
		\end{align}
		The hypothesis and the regularity setting imply that
		\begin{align}\label{error: NS multistep hypothesis deduced uniform bound}
			\mynormb{\nabla\myvec{u}^{\ell}}^2+\sum_{j=\rmk}^\ell\tau\mynormb{\Delta\myvec{u}^{j}}^2\le(1+\ck_{\myvec{u}})^2
			\quad\text{for $\rmk-1\le \ell\le m-1$}.
		\end{align}

		To derive the error estimate at $t_m$, we present some estimates on the pressure gradient and the nonlinear term. According to the equivalent equation \eqref{conti: Stokes pressure} for the Stokes pressure $p_{\mathrm{s}}$, the pressure error equation \eqref{scheme: perturbed multistep-pressure Poisson} has the following form
		\begin{align*}	
			\myinnerb{\nabla \tilde{p}^{\ell},\nabla \phi}&\,=\myinnerb{\nu \nabla p_{\mathrm{s}}(\tilde{\myvec{u}}^{\ell})-\tilde{\myvec{g}}^{\ell},\nabla \phi}\quad\text{for $0\le\ell\le N$, $\forall\phi\in H^1(\Omega)$.}
		\end{align*}
		Taking $\phi:=\tilde{p}^{\ell}$ arrives at
		\begin{align}\label{stability: multistep-pressure gradient}
			\mynormb{\nabla \tilde{p}^{\ell}}\le
			\nu \mynormb{\nabla p_{\mathrm{s}}(\tilde{\myvec{u}}^{\ell})}+\mynormb{\tilde{\myvec{g}}^{\ell}}\quad\text{for $0\le \ell\le N$,}
		\end{align}
		and thus
		\begin{align}\label{stability: multistep-pressure gradient2}
			\mynormb{\nabla \tilde{p}^{\ell}+\tilde{\myvec{g}}^{\ell}}\le
			\nu \mynormb{\nabla p_{\mathrm{s}}(\tilde{\myvec{u}}^{\ell})}+2\mynormb{\tilde{\myvec{g}}^{\ell}}\quad\text{for $0\le \ell\le N$.}
		\end{align}
		By using Lemma \ref{lem: nonlinear convective bound}, we can bound the  convection error term by
		\begin{align}\label{error: multistep hypothesis deduced nonlinear bound}
			\mynormb{\tilde{\myvec{g}}^{j}}^2
			\le&\,2\mynormb{\tilde{\myvec{u}}^{j}\cdot\nabla\myvec{u}^{j}}^2
			+2\mynormb{\myvec{U}^{j}\cdot\nabla\tilde{\myvec{u}}^{j}}^2\notag\\
			\le&\,	2\ck_l\mynormb{\nabla\tilde{\myvec{u}}^{j}}^2
			\mynormb{\nabla\myvec{u}^{j}}\mynormb{\Delta\myvec{U}^{j}-\Delta\tilde{\myvec{u}}^{j}}
			+2\ck_l\mynormb{\nabla\myvec{U}^{j}}^2\mynormb{\nabla\tilde{\myvec{u}}^{j}}\mynormb{\Delta\tilde{\myvec{u}}^{j}}
			\notag\\
			\le&\,	2\ck_l\mynormb{\nabla\tilde{\myvec{u}}^{j}}^2
			\mynormb{\nabla\myvec{u}^{j}}\mynormb{\Delta\myvec{U}^{j}}+2\ck_l\mynormb{\nabla\tilde{\myvec{u}}^{j}}^2
			\mynormb{\nabla\myvec{u}^{j}}\mynormb{\Delta\tilde{\myvec{u}}^{j}}
			+2\ck_l\ck_{\myvec{u}}^2\mynormb{\nabla\tilde{\myvec{u}}^{j}}\mynormb{\Delta\tilde{\myvec{u}}^{j}}
			\notag\\
			\le&\,	2\ck_l(1+\ck_{\myvec{u}})\ck_{\myvec{u}}\mynormb{\nabla\tilde{\myvec{u}}^{j}}^2
			+2\ck_l(1+\ck_{\myvec{u}})^2\mynormb{\nabla\tilde{\myvec{u}}^{j}}\mynormb{\Delta\tilde{\myvec{u}}^{j}}
			\notag\\
			\leq&\,\ck_1\epsilon_2^{-1}\mynormb{\nabla\tilde{\myvec{u}}^{j}}^2
			+\epsilon_2\mynormb{\Delta\tilde{\myvec{u}}^{j}}^2
			\quad\text{for $\rmk-1\le j\le m-1$},
		\end{align}
		where  $\ck_1:=2\ck_l(1+\ck_{\myvec{u}})\ck_{\myvec{u}}\epsilon_2+\ck_l^2(1+\ck_{\myvec{u}})^4$ with a small constant $\epsilon_2>0$ to be determined.

		{ By testing the error equation \eqref{scheme: perturbed multistep-momentum-differential}  with $-2\tau\Delta \tilde{\myvec{u}}^{i}$,}
		and summing over $i$ from $i=\rmk$ to $m$, we have
		\begin{align}\label{H1 stability: multistep momentum-difference-product}
			&\,	\mynormb{\nabla\tilde{\myvec{u}}^{m}}^2-\mynormb{\nabla\tilde{\myvec{u}}^{\rmk-1}}^2
			+\sum_{i=\rmk}^{m}\mynormb{\nabla(\tilde{\myvec{u}}^{i}-\tilde{\myvec{u}}^{i-1})}^2+
			2\nu\tau\sum_{i=\rmk}^{m}\sum_{j=\rmk}^{i}\hat{b}_{i-j}^{(\rmk)}\myinnerb{\Delta \tilde{\myvec{u}}^{j},\Delta \tilde{\myvec{u}}^{i}}		\notag\\
			&\,\hspace{1cm}	=2\tau\sum_{i=\rmk}^{m}\sum_{j=\rmk}^{i}\hat{c}_{i-j}^{(\rmk)}\myinnerb{\nabla \tilde{p}^{j-1}+\tilde{\myvec{g}}^{j-1},\Delta \tilde{\myvec{u}}^{i}}
			-2\tau\sum_{i=\rmk}^{m}\sum_{j=\rmk}^{i}a_{i-j}^{(-1,\rmk)}\myinnerb{\myvec{R}^{j},\Delta \tilde{\myvec{u}}^{i}}.
		\end{align}
		Lemma \ref{lemma: bound quadratic form} (i) gives
		\begin{align*}
			&\,2\nu\tau\sum_{i=\rmk}^{m}\sum_{j=\rmk}^{i}\hat{b}_{i-j}^{(\rmk)}\myinnerb{\Delta \tilde{\myvec{u}}^{j},\Delta \tilde{\myvec{u}}^{i}}\ge 2\lambda_{\mathrm{I}}^{(\rmk)}\nu\sum_{i=\rmk}^{m}\tau\mynormb{\Delta \tilde{\myvec{u}}^{i}}^2.
		\end{align*}
		Applying Lemma \ref{lemma: bound quadratic form} (iii) and the pressure gradient estimate \eqref{stability: multistep-pressure gradient2}, the first term at the right hand side (RHS) of
		\eqref{H1 stability: multistep momentum-difference-product} can be bounded by
		\begin{align*}
			\textrm{RHS}_1\le&\,
			2\sigma_{\mathrm{E}}^{(\rmk)}\sqrt{\sum_{i=\rmk}^{m}\tau\mynormb{\nabla \tilde{p}^{i-1}+\tilde{\myvec{g}}^{i-1}}^2}
			\sqrt{\sum_{i=\rmk}^{m}\tau\mynormb{\Delta \tilde{\myvec{u}}^{i}}^2}\\
			\le &\,
			4\sigma_{\mathrm{E}}^{(\rmk)}\sqrt{\sum_{i=\rmk}^{m}\tau\mynormb{\tilde{\myvec{g}}^{i-1}}^2}
			\sqrt{\sum_{i=\rmk}^{m}\tau\mynormb{\Delta \tilde{\myvec{u}}^{i}}^2}\\
			&\, +2\sigma_{\mathrm{E}}^{(\rmk)}\nu \sqrt{\sum_{i=\rmk}^{m}\tau\mynorm{\nabla p_{\mathrm{s}}(\tilde{\myvec{u}}^{i-1})}^2}
			\sqrt{\sum_{i=\rmk}^{m}\tau\mynormb{\Delta \tilde{\myvec{u}}^{i}}^2},
		\end{align*}
		where the triangular inequality was used in the last step. Thus, by using the Young inequality and the convection error bound \eqref{error: multistep hypothesis deduced nonlinear bound}, the right hand side (RHS) of \eqref{H1 stability: multistep momentum-difference-product} can be bounded by
		\begin{align*}
			\textrm{RHS}\le &\,2\epsilon_3\lambda_{\mathrm{I}}^{(\rmk)}\sum_{i=\rmk}^{m}\tau \mynormb{\Delta \tilde{\myvec{u}}^{i}}^2
			+\frac{2(\sigma_{\mathrm{E}}^{(\rmk)})^2}{\epsilon_3\lambda_{\mathrm{I}}^{(\rmk)}}
			\sum_{i=\rmk}^{m}\tau \mynormb{\tilde{\myvec{g}}^{i-1}}^2\\
			&\,+\frac{(\sigma_{\mathrm{E}}^{(\rmk)})^2\nu}{\lambda_{\mathrm{I}}^{(\rmk)}}
			\sum_{i=\rmk}^{m}\tau \mynormb{\nabla p_{\mathrm{s}}(\tilde{\myvec{u}}^{i-1})}^2
			+\lambda_{\mathrm{I}}^{(\rmk)}\nu \sum_{i=\rmk}^{m}\tau \mynormb{\Delta \tilde{\myvec{u}}^{i}}^2\\
			&\,+\frac{(\sigma_{\mathrm{F}}^{(\rmk)})^2}{\epsilon_3\lambda_{\mathrm{I}}^{(\rmk)}}\sum_{i=1}^{m}\tau \mynormb{\myvec{R}^{i}}^2
			+\epsilon_3\lambda_{\mathrm{I}}^{(\rmk)}\sum_{i=\rmk}^{m}\tau \mynormb{\Delta \tilde{\myvec{u}}^{i}}^2\\
			\le&\,(\nu +3\epsilon_3)\lambda_{\mathrm{I}}^{(\rmk)}\sum_{i=\rmk}^{m}\tau \mynormb{\Delta \tilde{\myvec{u}}^{i}}^2	
			+\frac{2(\sigma_{\mathrm{E}}^{(\rmk)})^2\epsilon_2}{\epsilon_3\lambda_{\mathrm{I}}^{(\rmk)}}
			\sum_{i=\rmk}^{m}\tau \mynormb{\Delta\tilde{\myvec{u}}^{i-1}}^2
			\\
			&\,	
			+\frac{2\ck_1(\sigma_{\mathrm{E}}^{(\rmk)})^2}{\epsilon_2\epsilon_3\lambda_{\mathrm{I}}^{(\rmk)}}
			\sum_{i=\rmk}^{m}\tau \mynormb{\nabla\tilde{\myvec{u}}^{i-1}}^2
			+\frac{(\sigma_{\mathrm{E}}^{(\rmk)})^2\nu}{\lambda_{\mathrm{I}}^{(\rmk)}}
			(\tfrac12+\varepsilon)\sum_{i=\rmk}^{m}\tau \mynormb{\Delta\tilde{\myvec{u}}^{i-1}}^2		\\
			&\,+
			\frac{\ck_{\varepsilon}\nu(\sigma_{\mathrm{E}}^{(\rmk)})^2}{\lambda_{\mathrm{I}}^{(\rmk)}}
			\sum_{i=\rmk}^{m}\tau \mynormb{\nabla\tilde{\myvec{u}}^{i-1}}^2	+\frac{(\sigma_{\mathrm{F}}^{(\rmk)})^2}{\epsilon_3\lambda_{\mathrm{I}}^{(\rmk)}}
			\sum_{i=\rmk}^{m}\tau \mynormb{\myvec{R}^{i}}^2,
		\end{align*}
		where we use the Stokes pressure estimate from Lemma \ref{lem: Stokes pressure bound}
		(by taking  $0<\varepsilon<\frac18$)
		\begin{align*}
			\mynormb{\nabla p_{\mathrm{s}}(\tilde{\myvec{u}}^{\ell})}^2\le (\tfrac12+\varepsilon)\mynormb{\Delta\tilde{\myvec{u}}^{\ell}}^2
			+\ck_{\varepsilon}\mynormb{\nabla\tilde{\myvec{u}}^{\ell}}^2\quad\text{for $0\le \ell\le N$}.
		\end{align*}
		Then it follows from \eqref{H1 stability: multistep momentum-difference-product}  that
		\begin{align*}
			\mynormb{\nabla\tilde{\myvec{u}}^{m}}^2-&\,\mynormb{\nabla\tilde{\myvec{u}}^{\rmk-1}}^2
			+(\nu -3\epsilon_3)\lambda_{\mathrm{I}}^{(\rmk)}\tau\mynormb{\Delta \tilde{\myvec{u}}^{m}}^2
			-(\nu -3\epsilon_3)\lambda_{\mathrm{I}}^{(\rmk)}\tau\mynormb{\Delta \tilde{\myvec{u}}^{\rmk-1}}^2 \notag\\
			+&\,\kbra{(\nu -3\epsilon_3)-\brab{	(\tfrac12+\varepsilon)\nu+2\epsilon_2\epsilon_3^{-1}}\frac{(\sigma_{\mathrm{E}}^{(\rmk)})^2}{(\lambda_{\mathrm{I}}^{(\rmk)})^2} }\lambda_{\mathrm{I}}^{(\rmk)}\sum_{i=\rmk}^{m-1}\tau \mynormb{\Delta\tilde{\myvec{u}}^{i}}^2\notag\\
			\le&\,
			\brab{\frac{2\ck_1}{\epsilon_2}
				+\ck_{\varepsilon}\epsilon_3\nu}
			\frac{(\sigma_{\mathrm{E}}^{(\rmk)})^2}{\epsilon_3\lambda_{\mathrm{I}}^{(\rmk)}}
			\sum_{i=\rmk}^{m}\tau \mynormb{\nabla\tilde{\myvec{u}}^{i-1}}^2
			+\frac{(\sigma_{\mathrm{F}}^{(\rmk)})^2}{\epsilon_3\lambda_{\mathrm{I}}^{(\rmk)}}
			\sum_{i=\rmk}^{m}\tau \mynormb{\myvec{R}^{i}}^2.
		\end{align*}
		Under the priori setting $\mathfrak{I}_{\mathrm{IE}}^{(\rmk)}>\frac{\sqrt{2}}{2}$, one can choose a small $\varepsilon:=\frac{(\lambda_{\mathrm{I}}^{(\rmk)})^2}{4(\sigma_{\mathrm{E}}^{(\rmk)})^2}-\frac18\in(0,\frac18)$  and $\epsilon_2:=\epsilon_3\varepsilon\nu/2$ such that
		$$(\nu -3\epsilon_3)-\brab{	(\tfrac12+\varepsilon)\nu+2\epsilon_2\epsilon_3^{-1}}\frac{(\sigma_{\mathrm{E}}^{(\rmk)})^2}{(\lambda_{\mathrm{I}}^{(\rmk)})^2}= \nu-3\epsilon_3-\frac{\tfrac12+2\varepsilon}{\frac12+4\varepsilon}\nu
		=\frac{2\varepsilon\nu}{\frac12+4\varepsilon}-3\epsilon_3.$$
		By taking $\epsilon_3:=\frac{\varepsilon\nu}{1+8\varepsilon}$, one has
		$$(\nu -3\epsilon_3)-\brab{	(\tfrac12+\varepsilon)\nu+2\epsilon_2\epsilon_3^{-1}}\frac{(\sigma_{\mathrm{E}}^{(\rmk)})^2}{(\lambda_{\mathrm{I}}^{(\rmk)})^2}=\epsilon_3\quad\text{and}\quad\nu -3\epsilon_3=\frac{\frac12+\tfrac52\varepsilon}{\frac12+4\varepsilon}\nu>\epsilon_3.$$
		Thus, with the constant $\ck_2:=\frac{4\ck_1}{\epsilon_3\varepsilon\nu}
		+\ck_{\varepsilon}\epsilon_3\nu$, we have
		\begin{align*}
			\mynormb{\nabla\tilde{\myvec{u}}^{m}}^2
			&\,	+\epsilon_3\lambda_{\mathrm{I}}^{(\rmk)}\sum_{i=\rmk}^{m}\tau \mynormb{\Delta \tilde{\myvec{u}}^{i}}^2
			\le\mynormb{\nabla\tilde{\myvec{u}}^{\rmk-1}}^2+\lambda_{\mathrm{I}}^{(\rmk)}\nu\tau\mynormb{\Delta \tilde{\myvec{u}}^{\rmk-1}}^2\\
			&\,+\frac{\ck_2(\sigma_{\mathrm{E}}^{(\rmk)})^2}{\epsilon_3\lambda_{\mathrm{I}}^{(\rmk)}}
			\sum_{i=\rmk-1}^{m-1}\tau \mynormb{\nabla\tilde{\myvec{u}}^{i}}^2
			+\frac{(\sigma_{\mathrm{F}}^{(\rmk)})^2}{\epsilon_3\lambda_{\mathrm{I}}^{(\rmk)}}\sum_{i=\rmk}^{m}\tau \mynormb{\myvec{R}^{i}}^2.
		\end{align*}
		The standard discrete Gr\"{o}nwall inequality yields
		\begin{align*}
			\mynormb{\nabla\tilde{\myvec{u}}^{m}}^2+
			\epsilon_3\lambda_{\mathrm{I}}^{(\rmk)}\sum_{i=\rmk}^{m}\tau \mynormb{\Delta \tilde{\myvec{u}}^{i}}^2
			&\,\le\frac{(\sigma_{\mathrm{F}}^{(\rmk)})^2}{\epsilon_3\lambda_{\mathrm{I}}^{(\rmk)}}
			\exp\braB{\frac{\ck_2(\sigma_{\mathrm{E}}^{(\rmk)})^2}{\epsilon_3\lambda_{\mathrm{I}}^{(\rmk)}} t_{m-\rmk+1}}
			\\&\,*\braB{\mynormb{\nabla\tilde{\myvec{u}}^{\rmk-1}}^2+\lambda_{\mathrm{I}}^{(\rmk)}\nu\tau\mynormb{\Delta \tilde{\myvec{u}}^{\rmk-1}}^2+\sum_{i=\rmk}^{m}\tau \mynormb{\myvec{R}^{i}}^2},
		\end{align*}
		and then, by using the prior setting \eqref{eq: initial and trunction error estimate},
		\begin{align}\label{error: multistep H1 norm}
			\mynormb{\nabla\tilde{\myvec{u}}^{m}}^2+
			\epsilon_3\lambda_{\mathrm{I}}^{(\rmk)}\sum_{i=\rmk}^{m}\tau \mynormb{\Delta \tilde{\myvec{u}}^{i}}^2\le \frac{(\sigma_{\mathrm{F}}^{(\rmk)})^2}{\epsilon_3\lambda_{\mathrm{I}}^{(\rmk)}}
			\tilde{\ck}_{\myvec{u}}^2\exp\braB{\frac{(\sigma_{\mathrm{E}}^{(\rmk)})^2}{\epsilon_3\lambda_{\mathrm{I}}^{(\rmk)}}\ck_2T_*}\tau^2.
		\end{align}
		It implies that, since $\epsilon_3\lambda_{\mathrm{I}}^{(\rmk)}<1$ according to Lemma \ref{lem: IELM stability intensity upper bound},
		\begin{align*}
			\mynormb{\nabla\tilde{\myvec{u}}^{m}}^2+
			\sum_{i=\rmk}^{m}\tau \mynormb{\Delta \tilde{\myvec{u}}^{i}}^2\le \ck_3^2\tau^2,
		\end{align*}
		where the constant $\ck_3:=\frac{\sigma_{\mathrm{F}}^{(\rmk)}}{\epsilon_3\lambda_{\mathrm{I}}^{(\rmk)}}
		\tilde{\ck}_{\myvec{u}}
		\exp\braB{\frac{(\sigma_{\mathrm{E}}^{(\rmk)})^2}{2\epsilon_3\lambda_{\mathrm{I}}^{(\rmk)}}\ck_2T_*}$.
		By choosing the small time-step size $\tau\le 1/\ck_3$,
		we recover the desired $\ell^{\infty}(H^1)\cap \ell^{2}(H^2)$ norm error bound \eqref{error: NS multistep H2 norm bound}  holds for $\ell=m$ and
		complete the mathematical induction.
		
		As a byproduct, the error bound \eqref{error: multistep H1 norm} holds for any $\rmk\le m\le N$.	
		By using the regularity assumption of solution, it is easy to obtain the following stability estimate
		\begin{align}\label{error: multistep H2 norm stability1}
			\mynormb{\nabla\myvec{u}^{m}}^2+\sum_{j=\rmk}^m\tau\mynormb{\Delta\myvec{u}^{j}}^2\le(1+\ck_{\myvec{u}})^2\quad\text{for $\rmk\le m\le N$.}
		\end{align}
	By using the pressure gradient estimate \eqref{stability: multistep-pressure gradient}, Lemma \ref{lem: Stokes pressure bound}
		and  the convection error bound \eqref{error: multistep hypothesis deduced nonlinear bound}, which is also valid for $\ell=m$ according to \eqref{error: multistep H2 norm stability1}, one applies the error estimate \eqref{error: multistep H1 norm} to obtain
		\begin{align*}
			&\,\sum_{i=\rmk}^{m}\tau \mynormb{\nabla \tilde{p}^{i}}^2
			\le
			2\nu\sum_{i=\rmk}^{m}\tau  \mynormb{\nabla p_{\mathrm{s}}(\tilde{\myvec{u}}^{i})}^2+2\sum_{i=\rmk}^{m}\tau \mynormb{\tilde{\myvec{g}}^{i}}^2\\
			&\,\hspace{1.3cm}\le\bra{\nu+2\varepsilon\nu+\epsilon_3\varepsilon\nu}\sum_{i=\rmk}^{m}\tau \mynormb{\Delta\tilde{\myvec{u}}^{i}}^2
			+2(\nu\ck_{\varepsilon}+\tfrac{2\ck_1}{\epsilon_3\varepsilon\nu})\sum_{i=\rmk}^{m}\tau \mynormb{\nabla\tilde{\myvec{u}}^{i}}^2\notag\\
			&\,\hspace{1.3cm}\le \kbra{\bra{\nu+2\varepsilon\nu+\epsilon_3\varepsilon\nu}+2(\nu\ck_{\varepsilon}+\tfrac{2\ck_1}{\epsilon_3\varepsilon\nu})T_*}\ck_3^2\tau^2\quad\text{for $1\le m\le N$.}\notag
		\end{align*}
		Thus, $\sum_{i=\rmk}^{m}\tau \mynormb{\nabla \tilde{p}^{i}}^2$ and  $\sum_{i=\rmk}^{m}\tau \mynormb{\nabla p^{i}}^2$  in the $\ell^{2}(H^1)$ norm are also bounded for $\rmk\le m\le N$ if the time-step size $\tau\le 1/\ck_3$. It completes the proof.
	\end{proof}
	
	By imposing further regularity in time, one has the following convergence result.
	
	\begin{theorem}\label{thm: NS multistep convergence}
		Assume that there exists a finite time $T_*>0$ such that the solution pair $(\myvec{u},p)$ of the  INSE model \eqref{cont: INSE-momentum}-\eqref{cont: INSE-incompressible} fulfills $\myvec{u}\in {C([0,T_*];H_0^1(\Omega)\cap H^2(\Omega))}$,  $\partial_{t}^{(\rmk+1)}\myvec{u}\in {L^2([0,T_*];L^2(\Omega))}$ and $\partial_{t}^{(\rmk)} p\in {L^2([0,T_*];H^1(\Omega))}$ for some integer $\rmk\ge1$.
		Assume further that the $\rmk$-step IELM method \eqref{scheme: general multistep-momentum} is consistent with an order of $O(\tau^{\rmk})$  and satisfies the assumptions of Lemma \ref{lemma: bound quadratic form}. If the controllability intensity $\mathfrak{I}_{\mathrm{IE}}^{(\rmk)}>\frac{\sqrt{2}}{2}$ and the step size $\tau$ is sufficiently small, the solution pairs $(\myvec{u}^{n},p^n)$ of the consistent splitting IELM methods \eqref{scheme: general multistep-pressure Poisson}-\eqref{scheme: general multistep-momentum} for the INSE model \eqref{cont: INSE-momentum}-\eqref{cont: INSE-incompressible} are unconditionally stable and convergent with an order of $O(\tau^{\rmk})$, provided the starting solutions $(\myvec{u}^{j},p^{j})$ for $1\le j\le \rmk-1$ are sufficiently accurate.
	\end{theorem}
	
	\section{Two classes of IELM methods}\label{sec: existing IELM}
	\setcounter{equation}{0}
	
	Among the four parameterized classes of IELM methods in \cite[Section 4]{LiaoQuanTangZhou:IMES}, the GBDF-$\rmk$ ($2\le \rmk\le5$) and SIELM-$\rmk$ ($2\le \rmk\le9$) schemes can theoretically maintain the unconditional stability of the consistent splitting IELM methods \eqref{scheme: general multistep-pressure Poisson}-\eqref{scheme: general multistep-momentum} because the associated controllability intensities $\mathfrak{I}_{\mathrm{IE}}^{(\rmk)}$ fulfill the stability requirement of Theorem \ref{thm: NS multistep stability}.
	
	\subsection{$\beta$-parameterized GBDF-$\rmk$ methods}
	The $\beta$-parameterized GBDF methods \cite{HuangShen:2024} are constructed by approximating each term of the differential equations at the off-set grid point $t_*:=t_{n-1+\beta}$  $(\beta\ge1)$, 
	\begin{align*}
		\sum_{j=0}^{\rmk-1}a_{\mathrm{G},j}^{(\rmk)}\partial_{\tau}u^{n-j}\approx u'(t_*),\quad
		\sum_{j=0}^{\rmk-1}b_{\mathrm{G},j}^{(\rmk)}u^{n-j}\approx u(t_*),\quad
		\sum_{j=0}^{\rmk-1}c_{\mathrm{G},j}^{(\rmk)}u^{n-j-1}\approx u(t_*).
	\end{align*}
	That is, the discrete coefficients $a_{\mathrm{G},j}^{(\rmk)}$, $b_{\mathrm{G},j}^{(\rmk)}$ and $c_{\mathrm{G},j}^{(\rmk)}$  for $0\le j\le \rmk-1$ can be determined independently
	by three linear algebraic systems of Vandermonde-type. 
	We recall the GBDF-$\rmk$ ($2\le \rmk\le5$) methods with the coefficient vectors as follows:
	\begin{itemize}[itemindent=0.4cm,leftmargin=1.5cm]
		\item[(GBDF2)]: $\vec{a}_{\mathrm{G}}^{(2)}=(\tfrac{1}{2}+\beta,\tfrac{1}{2}-\beta)$,
		$\vec{b}_{\mathrm{G}}^{(2)}=(\beta,1-\beta)$,
		$\vec{c}_{\mathrm{G}}^{(2)}=(\beta +1,-\beta)$;
		\item[(GBDF3)]: $\vec{a}_{\mathrm{G}}^{(3)}
		=\brab{\tfrac{3 \beta ^2+6 \beta +2}{6},\tfrac{-6 \beta ^2-6 \beta +5}{6},\tfrac{3 \beta ^2-1}{6}},$
		$\vec{b}_{\mathrm{G}}^{(3)}=\brab{\tfrac{\beta ^2+\beta}{2} ,1-\beta ^2,\tfrac{\beta ^2-\beta}{2}}$,\\
		$\vec{c}_{\mathrm{G}}^{(3)}=(\tfrac{\beta ^2+3 \beta +2}{2},-2\beta-\beta ^2,\tfrac{\beta ^2+\beta}{2})$;		
		\item[(GBDF4)]:\\$\vec{a}_{\mathrm{G}}^{(4)}
		=\brab{\tfrac{2 \beta ^3+9 \beta ^2+11 \beta +3}{12},\tfrac{-6 \beta ^3-21\beta ^2-9 \beta +13}{12}, \tfrac{6 \beta^3+15 \beta^2-3\beta-5}{12}, \tfrac{-2 \beta^3-3 \beta^2+\beta+1}{12}},$\\
		$\vec{b}_{\mathrm{G}}^{(4)}=\brab{\tfrac{\beta^3+3 \beta^2+2 \beta}{6},\tfrac{-\beta^3-2 \beta^2+\beta +2}{2},\tfrac{\beta^3+\beta^2 -2\beta}{2} ,\tfrac{\beta -\beta^3}{6}},$\\
		$\vec{c}_{\mathrm{G}}^{(4)}=\brab{\tfrac{\beta^3+6 \beta^2+11 \beta +6}{6}, \tfrac{-\beta^3-5 \beta^2 -6\beta}{2}, \tfrac{\beta^3+4\beta^2 +3\beta}{2}, \tfrac{-\beta^3-3 \beta^2-2\beta}{6} }$;
		\item[(GBDF5)]:
		{\begin{align*}
				\vec{a}_{\mathrm{G}}^{(5)}
				=&\,\big(\tfrac{5\beta ^4+40\beta ^3+105 \beta ^2+100 \beta+24}{120},\tfrac{-10 \beta ^4-70 \beta ^3-135 \beta ^2-25 \beta +77}{60},\big.\\
				&\,\big.\tfrac{15\beta ^4+90 \beta ^3+120 \beta ^2-45\beta-43}{60},
				\tfrac{-10 \beta ^4-50 \beta ^3-45 \beta ^2+25 \beta +17}{60},\tfrac{5 \beta ^4+20 \beta ^3+15 \beta ^2-10 \beta -6}{120}\big),\notag\\
				\vec{b}_{\mathrm{G}}^{(5)}=&\,\big(\tfrac{\beta(\beta ^3+6 \beta ^2+11 \beta +6)}{24},\tfrac{-\beta ^4-5 \beta ^3-5 \beta ^2+5 \beta +6}{6},\tfrac{\beta(\beta ^3+4 \beta ^2+\beta -6)}{4},\tfrac{\beta (-\beta ^3-3 \beta ^2+\beta +3)}{6}, \tfrac{\beta(\beta ^3+2 \beta ^2-\beta -2)}{24}\big),\notag\\
				\vec{c}_{\mathrm{G}}^{(5)}=&\,\big(\tfrac{\beta ^4+10 \beta ^3+35 \beta ^2+50 \beta +24}{24},
				-\tfrac{\beta  (\beta ^3+9 \beta ^2+26 \beta +24)}{6},\big.\notag\\
				&\,\qquad\big.\tfrac{\beta(\beta ^3+8 \beta ^2+19 \beta +12)}{4},
				-\tfrac{\beta(\beta ^3+7 \beta ^2+14 \beta +8)}{6},
				\tfrac{\beta (\beta ^3+6 \beta ^2+11 \beta +6)}{24}\big).\notag
		\end{align*}}
	\end{itemize}

\lan{It is easy to get the associated three characteristic polynomials
	\begin{align*}
		&\,\varrho_{a,\mathrm{G}}^{(\rmk)}(\zeta):=\sum_{j=0}^{\rmk-1}a^{(\rmk)}_{\mathrm{G},j}\zeta^{\rmk-1-j}=\sum_{j=1}^{\rmk}\frac{f_{\mathrm{G}}^{(j)}(1)}{j!}(\zeta-1)^{j-1}\quad \text{with}\quad f_{\mathrm{G}}(z)=z^{\rmk-1+\beta}\ln z,\\
		&\,\varrho_{b,\mathrm{G}}^{(\rmk)}(\zeta):=\sum_{j=0}^{\rmk-1}\tbinom{\rmk-2+\beta}{j}(\zeta-1)^j,\qquad
		\varrho_{c,\mathrm{G}}^{(\rmk)}(\zeta):=\sum_{j=0}^{\rmk-1}\tbinom{\rmk-1+\beta}{j}(\zeta-1)^j,
	\end{align*}
	where $\tbinom{m}{j}:=\frac{m(m-1)\cdots(m-j+1)}{j!}$ denotes the usual binomial coefficient. It is not difficulty to check that $\varrho_{a,\mathrm{G}}^{(\rmk)}(\zeta)$  is a Schur polynomial for $\rmk=2,3,4$ and 5 if  $\beta>0$, $\beta>\frac{\sqrt{21}-3}{6}$, $\beta>\sqrt{2}-1$ and $\beta>0.50969$, respectively. They ensure that the implicit parts of the GBDF-$\rmk$ ($2\le \rmk\le5$) methods are strongly $A(0)$ stable if $\beta\ge1$, while the implicit part of the GBDF-6 method is not strongly $A(0)$ stable for $\beta\ge2$, see \cite[Section 4]{LiaoQuanTangZhou:IMES}.}

	\begin{table}[htb!]
		\centering
		\begin{threeparttable}
			\centering
			\caption{Stability property of GBDF-$\rmk$ methods for $\beta\ge1$}
			\vspace*{0.3pt}
			\def\temptablewidth{1\textwidth}
			\label{table: controllability intensity GBDF}
			\begin{tabular*}{\temptablewidth}{@{\extracolsep{\fill}}ccccc}
				\toprule
				$\rmk$	&$\sigma_{\mathrm{F},\mathrm{G}}^{(\rmk,\beta)}$&$\sigma_{\mathrm{E},\mathrm{G}}^{(\rmk,\beta)}$ &$\lambda_{\mathrm{I},\mathrm{G}}^{(\rmk,\beta)}$
				&$\mathfrak{I}_{\mathrm{IE},\mathrm{G}}^{(\rmk,\beta)}$\\  \midrule		
				$2$  &1 	& $\frac{2\beta+1}{2\beta}$& $\frac{2\beta-1}{2\beta}$
				&$\tfrac{2\beta-1}{2\beta +1}$\\[4pt]
				$3$  &1& $\frac{6 \beta ^2+12 \beta +3}{6 \beta ^2+6 \beta -2}$& $\frac{6 \beta ^2-4}{6 \beta ^2+6 \beta -2}$
				& $\tfrac{6\beta ^2-4}{6\beta ^2+12\beta +3}$  \\[4pt]
				$4$ &$\frac{11\beta-1}{10\beta}$	&$\frac{4 \beta ^3+19\beta ^2+20 \beta +3}{4 (\beta ^3+3 \beta ^2+\beta -1)}$
				& $\frac{4 \beta ^3+5 \beta ^2-4 \beta -3}{4 (\beta ^3+3 \beta ^2+\beta -1)}$
				&  $\tfrac{4 \beta ^3+5 \beta ^2-4 \beta -3}{4 \beta ^3+19 \beta ^2+20 \beta+3}$ \\[4pt]
				$5^*$  &$\frac{20\beta-1}{10\beta}$& $\frac{5(2 \beta ^4+30 \beta ^3+32 \beta ^2+38 \beta +5)}{10 \beta ^4+60 \beta ^3+90 \beta ^2-32}$& $\frac{5(2 \beta ^4+\beta ^3+4 \beta ^2-4 \beta -2)}{10 \beta ^4+60 \beta ^3+90 \beta ^2-32}$
				& $\frac{2 \beta ^4+\beta ^3+4 \beta ^2-4 \beta -2}{2 \beta ^4+30 \beta ^3+32 \beta ^2+38 \beta +5}$\\[4pt]
				$5^\star$  &$\frac{20\beta-1}{10\beta}$& $\frac{5 (2 \beta ^4+21 \beta ^3+29 \beta ^2+38 \beta +5)}{10 \beta ^4+60 \beta ^3+90 \beta ^2-32}$& $\frac{5 (2 \beta ^4+3 \beta ^3+25 \beta ^2-9 \beta -3)}{10 \beta ^4+60 \beta ^3+90 \beta ^2-32}$
				& $\frac{2 \beta ^4+3 \beta ^3+25 \beta ^2-9 \beta -3}{2 \beta ^4+21 \beta ^3+29 \beta ^2+38 \beta +5}$
			\end{tabular*}
			{\rule{\temptablewidth}{0.5pt}}
			\tnote{The two cases $\rmk=5^*$ and $5^\star$ require $1\le \beta<18$ and $\beta\ge18$, respectively.}
		\end{threeparttable}
	\end{table}	
	
	Table \ref{table: controllability intensity GBDF} collects the upper bounds of  $\sigma_{\mathrm{F},\mathrm{G}}^{(\rmk,\beta)}$ and $\sigma_{\mathrm{E},\mathrm{G}}^{(\rmk,\beta)}$, the lower bounds of $\lambda_{\mathrm{I},\mathrm{G}}^{(\rmk,\beta)}$ and $\mathfrak{I}_{\mathrm{IE},\mathrm{G}}^{(\rmk,\beta)}$ for the GBDF-$\rmk$ methods with the parameter $\beta\ge1$, see \cite[Table 2]{LiaoQuanTangZhou:IMES}.
	Simple computations show that 
	\begin{align*}
		&\mathfrak{I}_{\mathrm{IE},\mathrm{G}}^{(2,\beta)}>\frac{\sqrt{2}}{2}\quad\text{if $\beta>\frac{3}{2}+\sqrt{2}\approx 2.914$,}
		&\mathfrak{I}_{\mathrm{IE},\mathrm{G}}^{(3,\beta)}
		&>\frac{\sqrt{2}}{2}\quad\text{if $\beta>5.46572$,}\\
		&\mathfrak{I}_{\mathrm{IE},\mathrm{G}}^{(4,\beta)}>\frac{\sqrt{2}}{2}\quad\text{if $\beta>8.97865$,}
		&\mathfrak{I}_{\mathrm{IE},\mathrm{G}}^{(5,\beta)}
		&>\frac{\sqrt{2}}{2}\quad\text{if $\beta>19.9988$.}
	\end{align*}
	They say that, by choosing the parameters $\beta_{\rmk}=3,11/2, 9, 20$ corresponding to $\rmk=2,3,4,5$, respectively, the GBDF-$\rmk$ schemes theoretically maintain the unconditional stability of the consistent splitting IELM methods \eqref{scheme: general multistep-pressure Poisson}-\eqref{scheme: general multistep-momentum}. In this sense, Theorems \ref{thm: NS multistep stability} and \ref{thm: NS multistep convergence} essentially improve the results of \cite[Theorem 4.1]{HuangShen:2025mcom}, where three concrete cases with fixed parameters $\beta_{\rmk}=3,6,9$ corresponding to $\rmk=2,3,4$, respectively, were verified theoretically.
	
	\subsection{$\gamma$-parameterized SIELM-$\rmk$ methods}
	\lan{The SIELM-$\rmk$ ($2\le\rmk\le9$) methods \cite{LiaoQuanTangZhou:IMES} with the discrete coefficients $a_{\mathrm{S},j}^{(\rmk)}$, $b_{\mathrm{S},j}^{(\rmk)}$ and $c_{\mathrm{S},j}^{(\rmk)}$ are determined by the following three characteristic polynomials 
	$${\varrho}_{a,\mathrm{S}}^{(\rmk)}(\zeta):=\sum_{j=1}^{\rmk}\frac{f_{\mathrm{S}}^{(j)}(1)}{j!}(\zeta-1)^{j-1}\quad \text{with}\quad f_{\mathrm{S}}(z)=(\gamma z-\gamma+1)^{\rmk-1}z\ln z,$$ 
	$\varrho_{b,\mathrm{S}}^{(\rmk)}(\zeta):=\zeta(\gamma\zeta-\gamma+1)^{\rmk-1}$ and $\varrho_{c,\mathrm{S}}^{(\rmk)}(\zeta):=\zeta(\gamma\zeta-\gamma+1)^{\rmk-1}-\gamma^{\rmk-1}(\zeta-1)^{\rmk}$,
	respectively. By the Routh-Hurwitz criterion,  one can check that $\varrho_{a,\mathrm{S}}^{(\rmk)}(\zeta)$  is a Schur polynomial if  $\gamma>0$, $\gamma>\frac{1}{\sqrt{6}}$, $\gamma>0$, $\gamma >0.647518$, $\gamma >0.830364$, $\gamma >1.02816$, 
	$\gamma>1.23687$ and $\gamma>1.45371$ corresponding to the order index $\rmk=2,3,\cdots,8$ and 9, respectively, which ensure that the implicit parts of the SIELM-$\rmk$ methods for $2\le\rmk\le9$ are strongly $A(0)$-stable.}

	Here we recall the SIELM-$\rmk$ methods for $2\le\rmk\le6$ with the first coefficient vector:
	\begin{itemize}[itemindent=0.4cm,leftmargin=1.5cm]
		\item[(SIELM2)]: $\vec{a}_{\mathrm{S}}^{(2)}=(\tfrac{1}{2}+\gamma,\tfrac{1}{2}-\gamma)$;
		\item[(SIELM3)]: $\vec{a}_{\mathrm{S}}^{(3)}
		=\brab{\gamma ^2+\gamma -\tfrac{1}{6},\tfrac{5}{6}-2 \gamma ^2,\gamma ^2-\gamma +\tfrac{1}{3}};$
		\item[(SIELM4)]:
		\begin{align*}
			&\vec{a}_{\mathrm{S}}^{(4)}
			=\big(\gamma ^3+\tfrac{3 \gamma ^2}{2}-\tfrac{\gamma }{2}+\tfrac{1}{12},-3 \gamma ^3-\tfrac{3 \gamma ^2}{2}+3 \gamma -\tfrac{5}{12},3 \gamma ^3-\tfrac{3 \gamma ^2}{2}-\tfrac{3 \gamma }{2}+\tfrac{13}{12},-\gamma ^3+\tfrac{3 \gamma ^2}{2}-\gamma +\tfrac{1}{4}\big);
		\end{align*}
		\item[(SIELM5)]:
		\begin{align*}
			&\vec{a}_{\mathrm{S}}^{(5)}
			=\big(\gamma ^4+2 \gamma ^3-\gamma ^2+\tfrac{\gamma }{3}-\tfrac{1}{20},-4 \gamma ^4-4 \gamma ^3+7 \gamma ^2-2 \gamma +\tfrac{17}{60},\\
			&\qquad\qquad\big.6 \gamma ^4-9 \gamma ^2+6 \gamma -\tfrac{43}{60},
			-4 \gamma ^4+4 \gamma ^3+\gamma ^2-\tfrac{10 \gamma }{3}+\tfrac{77}{60},\gamma ^4-2 \gamma ^3+2 \gamma ^2-\gamma +\tfrac{1}{5}\big);
		\end{align*}
			\item[(SIELM6)]:
		\begin{align*}
			&\vec{a}_{\mathrm{S}}^{(6)}
			=\big(\gamma ^5+\tfrac{5 \gamma ^4}{2}-\tfrac{5 \gamma ^3}{3}+\tfrac{5 \gamma ^2}{6}-\tfrac{\gamma }{4}+\tfrac{1}{30},
			-5 \gamma ^5-\tfrac{15 \gamma ^4}{2}+\tfrac{40 \gamma ^3}{3}-\tfrac{35 \gamma ^2}{6}+\tfrac{5 \gamma }{3}-\tfrac{13}{60},\notag\\
			&\qquad10 \gamma ^5+5 \gamma ^4-\tfrac{80 \gamma ^3}{3}+20 \gamma ^2-5 \gamma +\tfrac{37}{60},-10 \gamma ^5+5 \gamma ^4+\tfrac{50 \gamma ^3}{3}-\tfrac{70 \gamma ^2}{3}+10 \gamma -\tfrac{21}{20},\notag\\
			&\qquad5 \gamma ^5-\tfrac{15 \gamma ^4}{2}+\tfrac{5 \gamma ^3}{3}+\tfrac{35 \gamma ^2}{6}-\tfrac{65 \gamma }{12}+\tfrac{29}{20}, -\gamma ^5+\tfrac{5 \gamma ^4}{2}-\tfrac{10 \gamma ^3}{3}+\tfrac{5 \gamma ^2}{2}-\gamma +\tfrac{1}{6}\big).
		\end{align*}
	\end{itemize}
	
\begin{table}[htb!]
	\centering
	\begin{threeparttable}
		\centering
		\caption{Stability property of SIELM-$\rmk$ methods}
		\vspace*{0.3pt}
		\def\temptablewidth{1\textwidth}
		\label{table: controllability intensity SIELM}
		\begin{tabular*}{\temptablewidth}{@{\extracolsep{\fill}}clc}
			\toprule
			$\rmk$&Range of $\gamma$ 
			&$\mathfrak{I}_{\mathrm{IE},\mathrm{S}}^{(\rmk,\gamma)}$\\  \midrule		
			$2$  &$[1,+\infty)$
			&$\tfrac{2\gamma-1}{2\gamma +1}$\\[5pt]
			$3$  &  $[1,+\infty)$
			& $\frac{(2\gamma-1)^2}{4 \gamma ^2+4 \gamma -1}$  \\[5pt]
			$4$ &	 $[6/5,+\infty)$		
			&  $\frac{(2 \gamma -1)^3}{8 \gamma ^3+12 \gamma ^2-6 \gamma +1}$ \\[5pt]
			$5$ & 	 $[7/5,+\infty)$
			& $\frac{(2\gamma-1)^4}{16 \gamma ^4+32 \gamma ^3-24 \gamma ^2+8 \gamma -1}$\\[5pt]
			$6$ &$[10/5,80]$
			& $\frac{(2 \gamma -1)^5}{32 \gamma ^5+80 \gamma ^4-80 \gamma ^3+40 \gamma ^2-10 \gamma +1}$
			\\[5pt]
			$7$ &$[11/5,60]$
			& {$\frac{(2 \gamma-1 )^6}{64 \gamma ^6+192 \gamma ^5-240 \gamma ^4+160 \gamma ^3-60 \gamma ^2+12 \gamma -1}$}
			\\[5pt]
			$8$ &$[13/5,40]$
			& {  $\frac{(2 \gamma -1)^7}{128 \gamma ^7+448 \gamma ^6-672 \gamma ^5+560 \gamma ^4-280 \gamma ^3+84 \gamma ^2-14 \gamma +1}$}\\[5pt]
			$9$ &$[15/5,25]$
			& {$ \frac{(2\gamma-1)^8}{256\gamma^8 + 1024\gamma^7 - 1792\gamma^6 + 1792\gamma^5 - 1120\gamma^4 + 448\gamma^3 - 112\gamma^2 + 16\gamma - 1}$}
		\end{tabular*}
		{\rule{\temptablewidth}{0.5pt}}
		\footnotesize
	\end{threeparttable}
\end{table}	
	
\lan{For the semi-generating functions $a_{\mathrm{S}}^{(\rmk)}(\theta)$, $b_{\mathrm{S}}^{(\rmk)}(\theta)$
and $c_{\mathrm{S}}^{(\rmk)}(\theta)$, defined via \eqref{matrix: A_L B_L C_L generating function},
of the SIELM-$\rmk$ methods, we have the following results, see \cite[Propositions SM3.1-SM3.17]{LiaoQuanTangZhou:IMES},
\begin{align}\label{intensity: lambda sigma-SIELM}
	\sigma_{\mathrm{F},\mathrm{S}}^{(\rmk)}=\frac{1}{\absb{a_{\mathrm{S}}^{(\rmk)}(0)}}=1,\quad
	\sigma_{\mathrm{E},\mathrm{S}}^{(\rmk)}=
	\frac{\absb{c_{\mathrm{S}}^{(\rmk)}(\pi)}}{\absb{a_{\mathrm{S}}^{(\rmk)}(\pi)}},	\quad
	\lambda_{\mathrm{I},\mathrm{S}}^{(\rmk)}=\Re\kbra{\frac{b_{\mathrm{S}}^{(\rmk)}(\pi)}{a_{\mathrm{S}}^{(\rmk)}(\pi)}}
\end{align}
if the free parameter $\gamma$ satisfies $\gamma_{\rmk}\le \gamma\le \gamma_{\rmk}^*$,  cf. Table \ref{table: controllability intensity SIELM}. }	\lan{Table \ref{table: controllability intensity SIELM} collects the value of    $\mathfrak{I}_{\mathrm{IE},\mathrm{S}}^{(\rmk,\gamma)}$ with the ranges of $\gamma$ for the  SIELM-$\rmk$ methods, see \cite[Table 3]{LiaoQuanTangZhou:IMES}. It is easy to check that
	the implicit-explicit controllability intensities for $2\le\rmk\le9$, respectively,
	\begin{align*}
		&\mathfrak{I}_{\mathrm{IE},\mathrm{S}}^{(2,\gamma)}>\frac{\sqrt{2}}{2}\quad\text{if $\gamma>2.914$,}
		&\mathfrak{I}_{\mathrm{IE},\mathrm{S}}^{(3,\gamma)}
			&>\frac{\sqrt{2}}{2}\quad\text{if $\gamma>5.56667$,}\\
		&\mathfrak{I}_{\mathrm{IE},\mathrm{S}}^{(4,\gamma)}>\frac{\sqrt{2}}{2}\quad\text{if $\gamma>8.22174$,}
		&\mathfrak{I}_{\mathrm{IE},\mathrm{S}}^{(5,\gamma)}
		&>\frac{\sqrt{2}}{2}\quad\text{if $\gamma>10.8775$,}\\
		&\mathfrak{I}_{\mathrm{IE},\mathrm{S}}^{(6,\gamma)}>\frac{\sqrt{2}}{2}\quad\text{if $13.5334<\gamma\le 80$,}
		&\mathfrak{I}_{\mathrm{IE},\mathrm{S}}^{(7,\gamma)}
		&>\frac{\sqrt{2}}{2}\quad\text{if $16.18966<\gamma\le 60$,}\\
			&\mathfrak{I}_{\mathrm{IE},\mathrm{S}}^{(8,\gamma)}
			>\frac{\sqrt{2}}{2}\quad\text{if $18.84576<\gamma\le 40$,}
	&\mathfrak{I}_{\mathrm{IE},\mathrm{S}}^{(9,\gamma)}
		&>\frac{\sqrt{2}}{2}\quad\text{if $21.5026<\gamma\le 25$.}
	\end{align*}
	 They say that, by choosing the parameters $\gamma_{\rmk}=3.0,5.6, 8.3, 10.9, 13.6, 16.2, 18.9$ and $21.6$ corresponding to $\rmk=2,3,\cdots,8$ and 9, respectively, the SIELM-$\rmk$ schemes theoretically maintain the unconditional stability of the consistent splitting IELM methods \eqref{scheme: general multistep-pressure Poisson}-\eqref{scheme: general multistep-momentum}.}

	\section{Numerical experiments}\label{sec: numerical example}
	\setcounter{equation}{0}
	
	This section presents some numerical tests of
	the GBDF-$\rmk$ and SIELM-$\rmk$ methods for the unconstrained INSE system \eqref{cont: unconstrained INSE-pressure Poisson}-\eqref{cont: unconstrained INSE-momentum}. For the convenience of comparisons, our numerical setting follows closely the benchmark
	tests in \cite{HuangShen:2025mcom}, where three specific cases $\beta=3,6$ and $9$ of the following GBDF-$\rmk$  schemes were verified for $\mathrm{k}=2,3$ and $4$, respectively,
	\begin{align*}	
		\myinner{\nabla p^{n},\nabla \phi}=&\,\myinner{
			\nu\Delta\myvec{u}^{n}-\nu\nabla\nabla\cdot\myvec{u}^{n}-\myvec{g}^{n},\nabla \phi}\quad\text{for $\forall\phi\in H^1(\Omega)$, $n\ge0$,}\\
		\sum_{j=0}^{\rmk-1}a_{\mathrm{G},j}^{(\rmk)}\partial_{\tau}\myvec{u}^{n-j}
		=&\,\nu\sum_{j=0}^{\rmk-1}b_{\mathrm{G},j}^{(\rmk)}\Delta\myvec{u}^{n-j}
		-\sum_{j=0}^{\rmk-1}c_{\mathrm{G},j}^{(\rmk)}\nabla p^{n-j-1}-\myvec{g}\braB{\sum_{j=0}^{\rmk-1}c_{\mathrm{G},j}^{(\rmk)}\myvec{u}^{n-j-1}}
	\end{align*}
	for $\myvec{x}\in\Omega$ with $\myvec{u}^n=0$ on $\partial\Omega$ and $n\ge\rmk$. Note that, the GBDF approximation in \cite{HuangShen:2025mcom} of the momentum equation is slightly different from \eqref{scheme: general multistep-momentum} on the explicit approximation of nonlinear convective term.
	
	\begin{figure}[htb!]
		\centering
		\subfigure[GBDF3]{\includegraphics[width=0.42\textwidth]{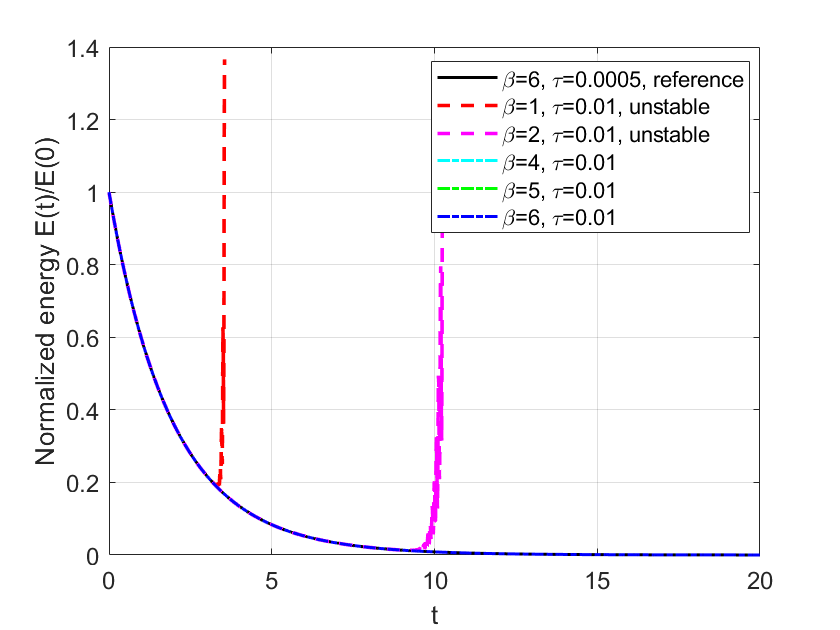}}
		\subfigure[GBDF4]{\includegraphics[width=0.42\textwidth]{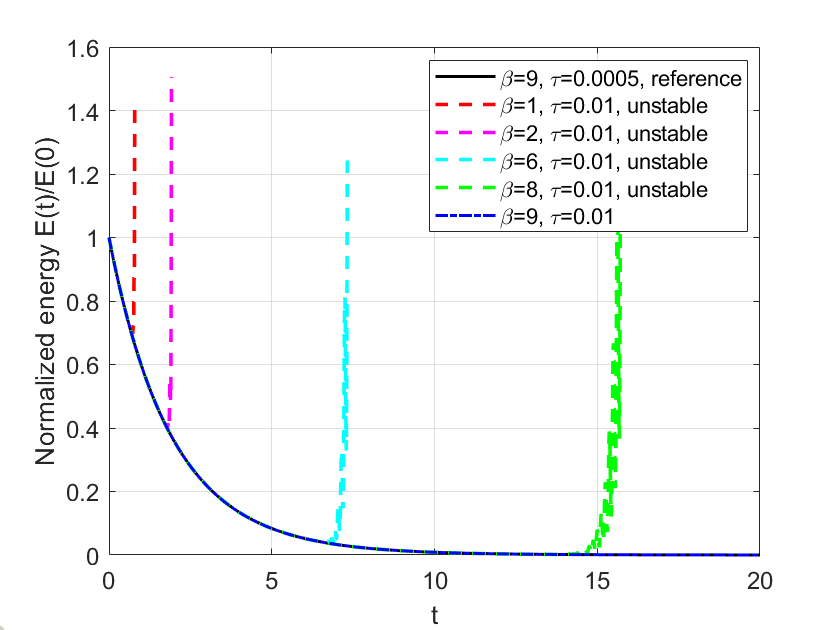}}
		\caption{Energy curves for the Stokes problem by GBDF3 and GBDF4 schemes.}\label{fig:stokes_GBDF}
	\end{figure}
	
	\begin{figure}[htb!]
		\centering
		\subfigure[SIELM3]{\includegraphics[width=0.42\textwidth]{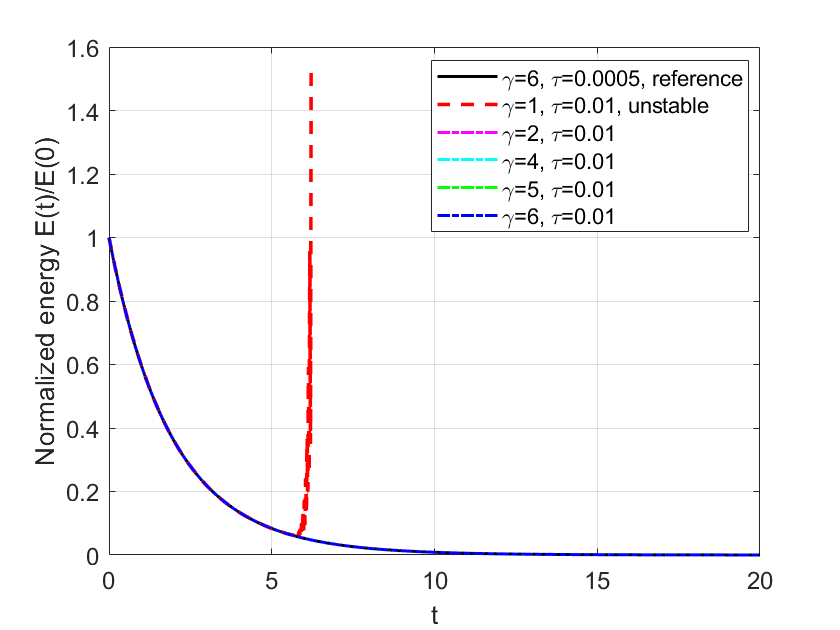}}
			\subfigure[SIELM4]{\includegraphics[width=0.42\textwidth]{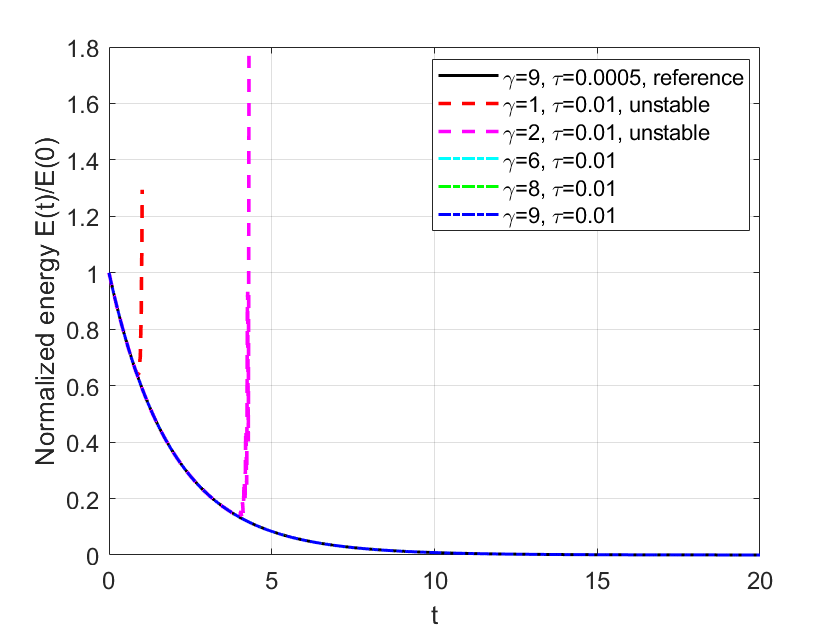}}
		\caption{Energy curves for the Stokes problem by SIELM3 and SIELM4 schemes.}\label{fig:stokes_SIELM}
	\end{figure}

	\begin{figure}[htb!]
		\centering
		\subfigure[Normalized energy evolution]{\includegraphics[width=0.42\textwidth]{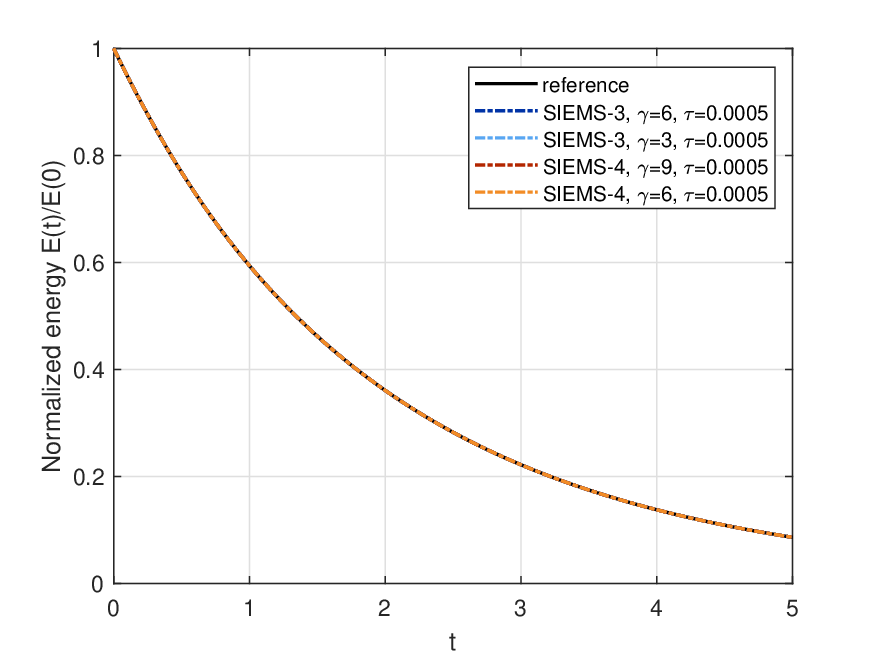}}
		\subfigure[T=0.01]{\includegraphics[width=0.42\textwidth]{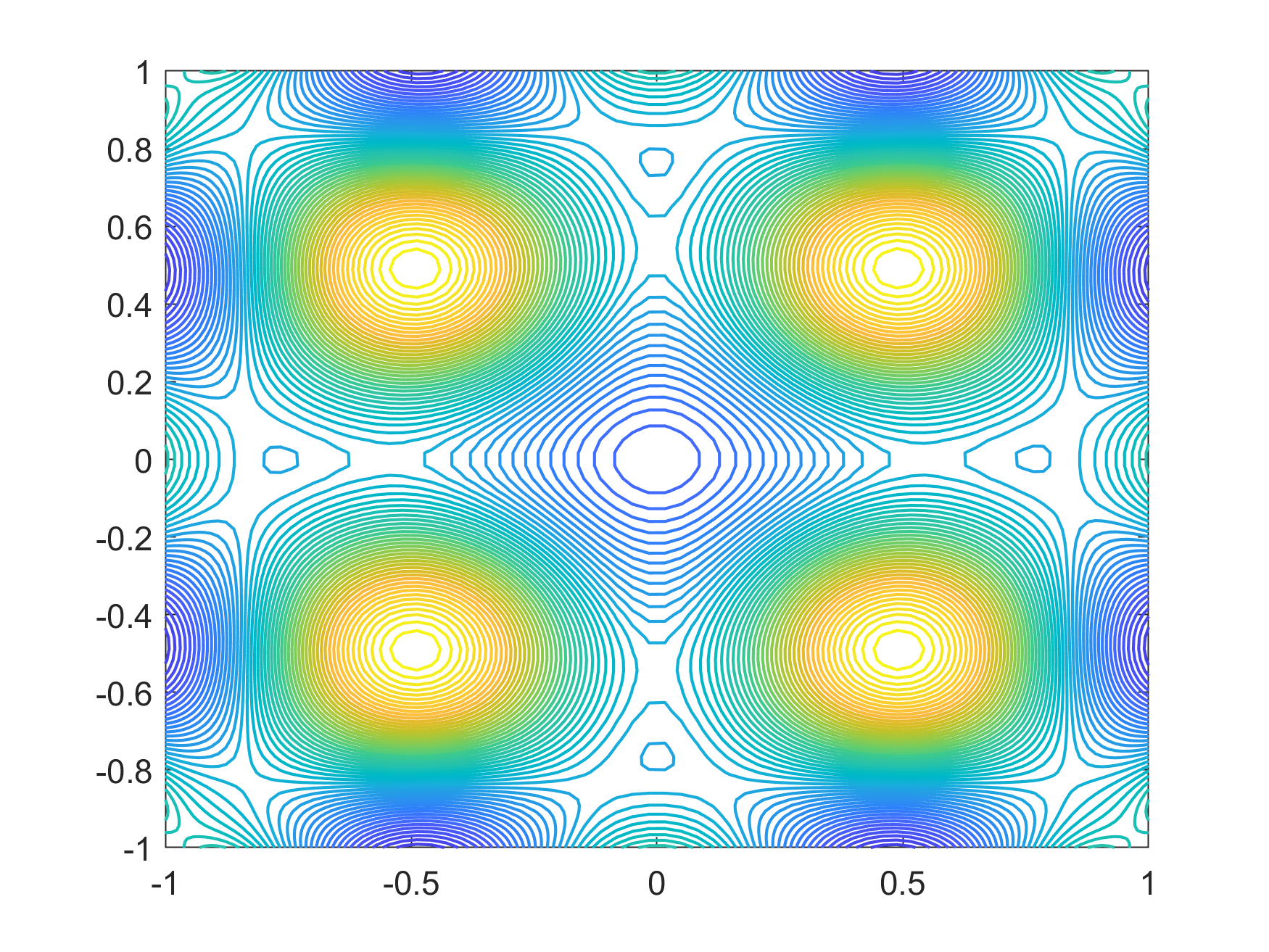}}\\
		\subfigure[T=3]{\includegraphics[width=0.42\textwidth]{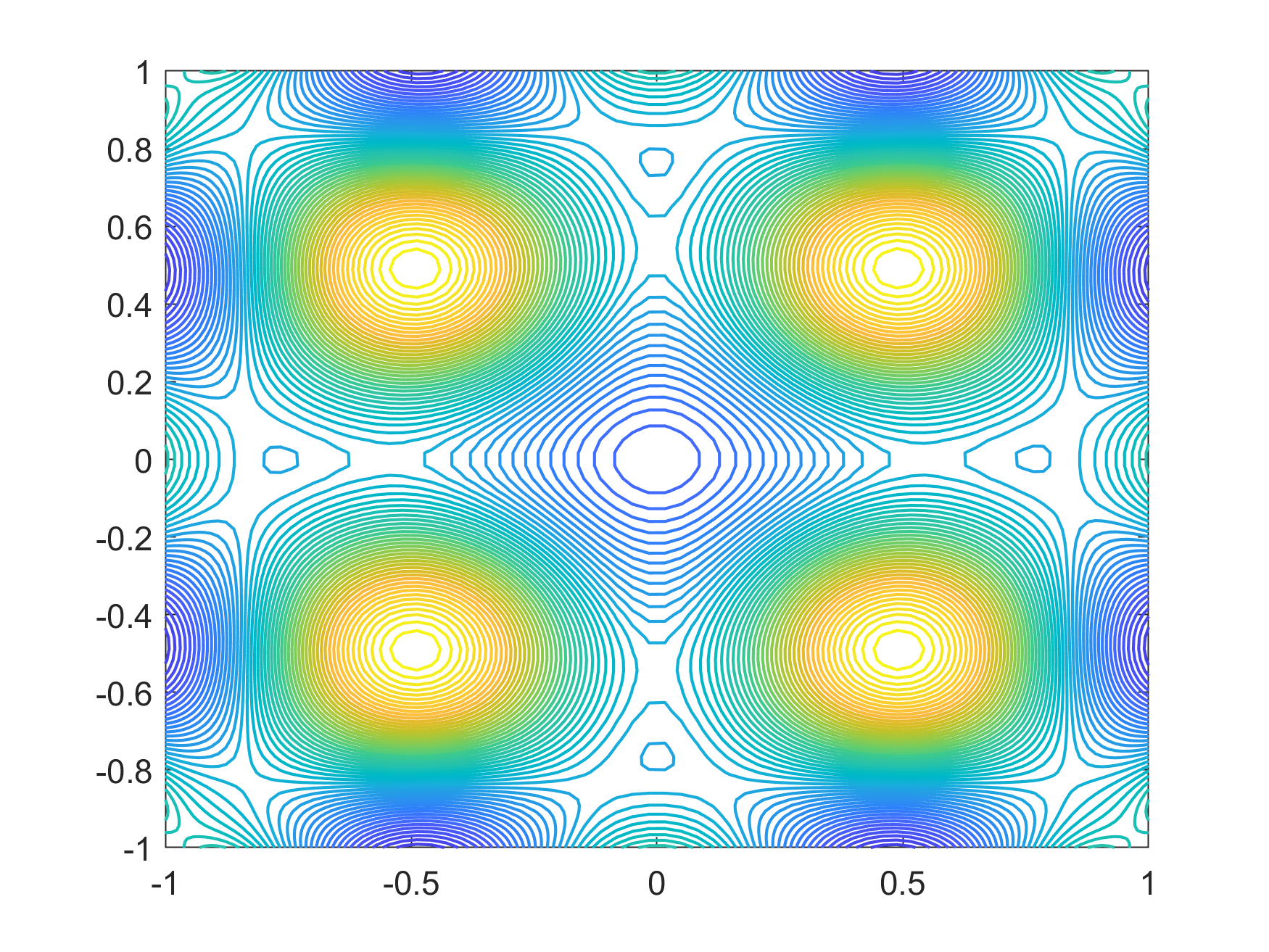}}
		\subfigure[T=5]{\includegraphics[width=0.42\textwidth]{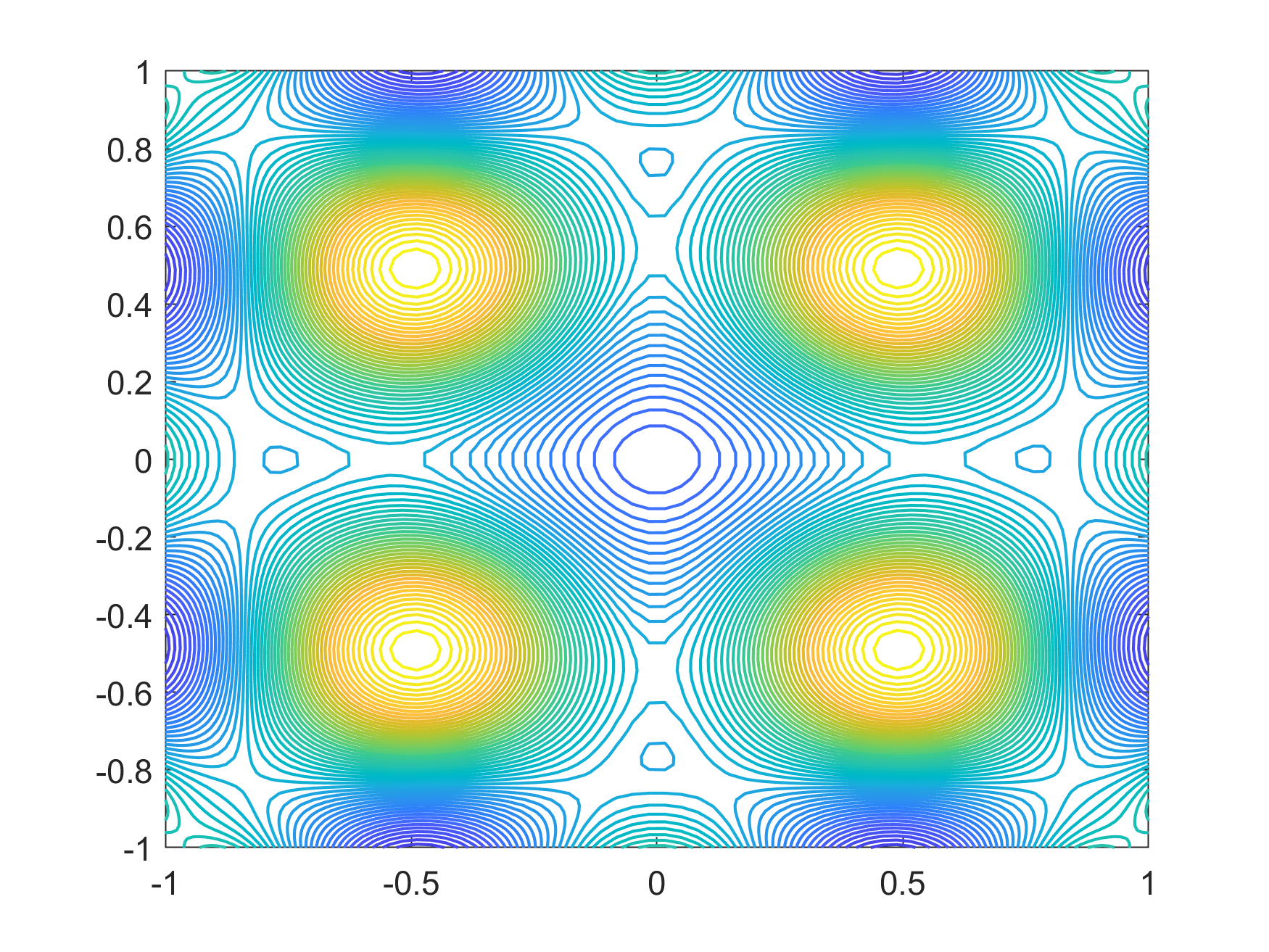}}
		\caption{Normalized energy evolution for the Navier-Stokes equations and
			snapshots of the vorticity contours at $T=0.01$, $T=3$ and $T=5$.}\label{fig:vorticity}
	\end{figure}
	
	To demonstrate the practical effects of high-order IELM methods, we adopt a Legendre-Galerkin spectral method in space and choose sufficiently many spatial modes so that the spatial discretization error is negligible compared with the temporal error. Our proof requires that the domain $\Omega$ has a $C^3$ boundary, see \cite[Theorem 1.2]{LiuLiuPego:2007} or Lemma \ref{lem: Stokes pressure bound}, in order to apply the Stokes pressure estimate, but the present numerical tests use the square domain $\Omega=(-1,1)^2$ with a rough $C^0$ boundary to empirically probe whether the theory can be extended to polygonal domains (though it remains open to us).

	
	\begin{example}\cite[Example 1]{HuangShen:2025mcom}
		We first consider the Stokes problem, namely the unconstrained INSE system \eqref{cont: unconstrained INSE-pressure Poisson}-\eqref{cont: unconstrained INSE-momentum} with $\myvec{g}=\myvec{0}$. The viscosity is set to $\nu=0.005$, and
		the initial velocity $ \myvec{u}^0=\brab{\sin2\pi y\sin^2\pi x,-\sin2\pi x\sin^2\pi y}.$
		The Legendre-Galerkin method is used in space with the mode number $N_x=N_y=128$.
	\end{example}

	Fig.~\ref{fig:stokes_GBDF} shows the energy curves, normalized by $E(0)$, generated by the GBDF3 and GBDF4 schemes with a properly large step $\tau=0.01$ for different parameters $\beta$, in which the reference energy is obtained by using a small step $\tau=0.0005$. As  confirmed  in \cite{HuangShen:2025mcom}, the choices $\beta=6$ and $\beta=9$ arrive at stable  solutions for the GBDF3 and GBDF4 schemes, respectively; while the smaller parameters $\beta=1$ and 2 always generate unstable solutions. These cases also support our theory since they satisfy our theoretical requirements $\beta >5.46572$ and $\beta>8.97865$ for the two schemes.    To check the sharpness of our stability indicator $\mathfrak{I}_{\mathrm{IE},\mathrm{G}}^{\mathrm{(k)}}>\sqrt{2}/2$, some smaller parameters $\beta$ are also examined in our tests. We observe from Fig.~\ref{fig:stokes_GBDF} that the proposed indicator would be somewhat sharp for the GBDF4 scheme since the parameters $\beta=6$ and 8 always generate unstable solutions, while the stability condition for the GBDF3 scheme  is far away from sharp, at least for this specific example, since the choices $\beta=4$ and 5 arrive at stable solutions.

	\begin{figure}[htb!]
		\centering
		\subfigure[GBDF-2]{\includegraphics[width=0.42\textwidth]{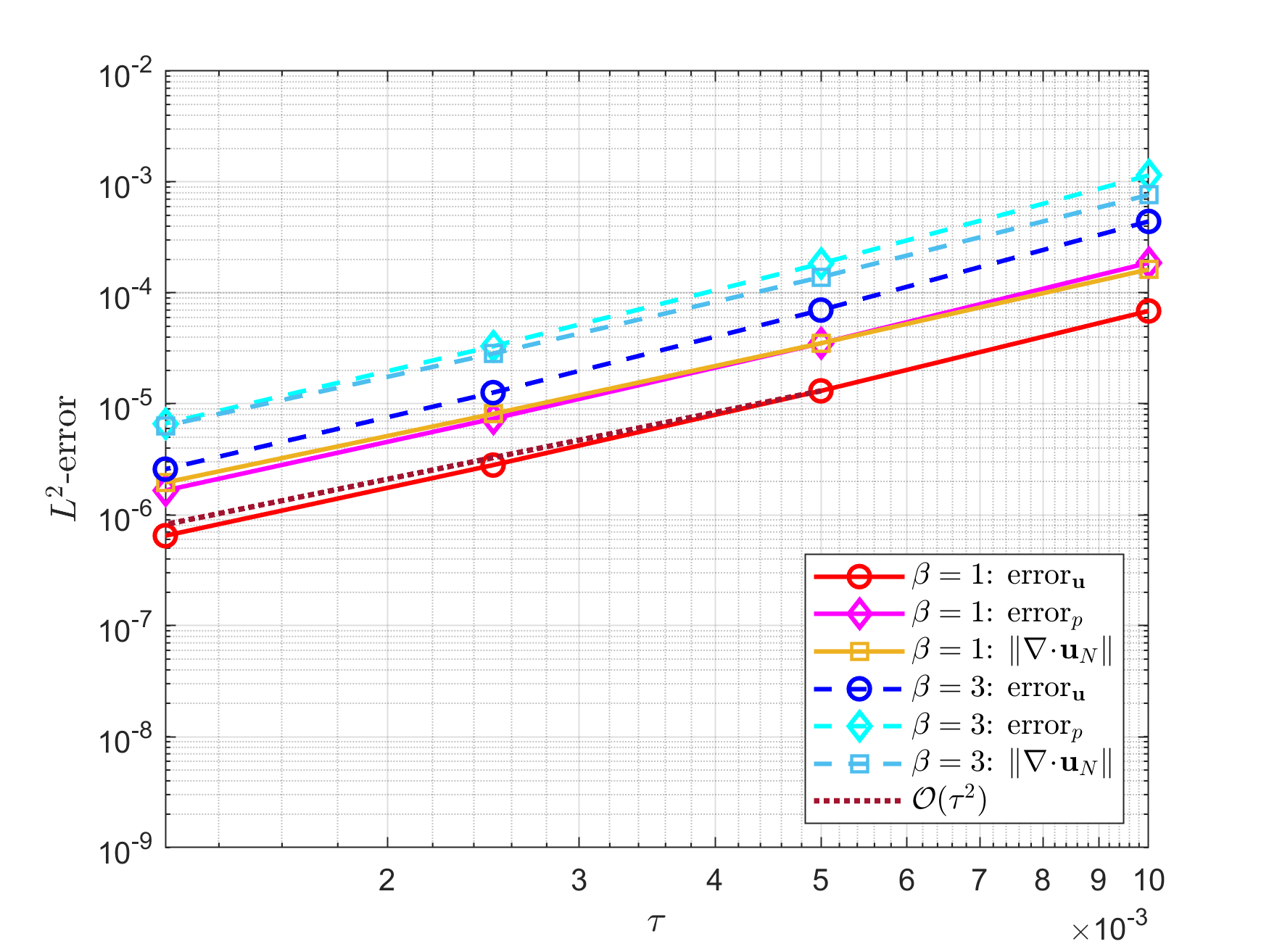}}
		\subfigure[GBDF-3]{\includegraphics[width=0.42\textwidth]{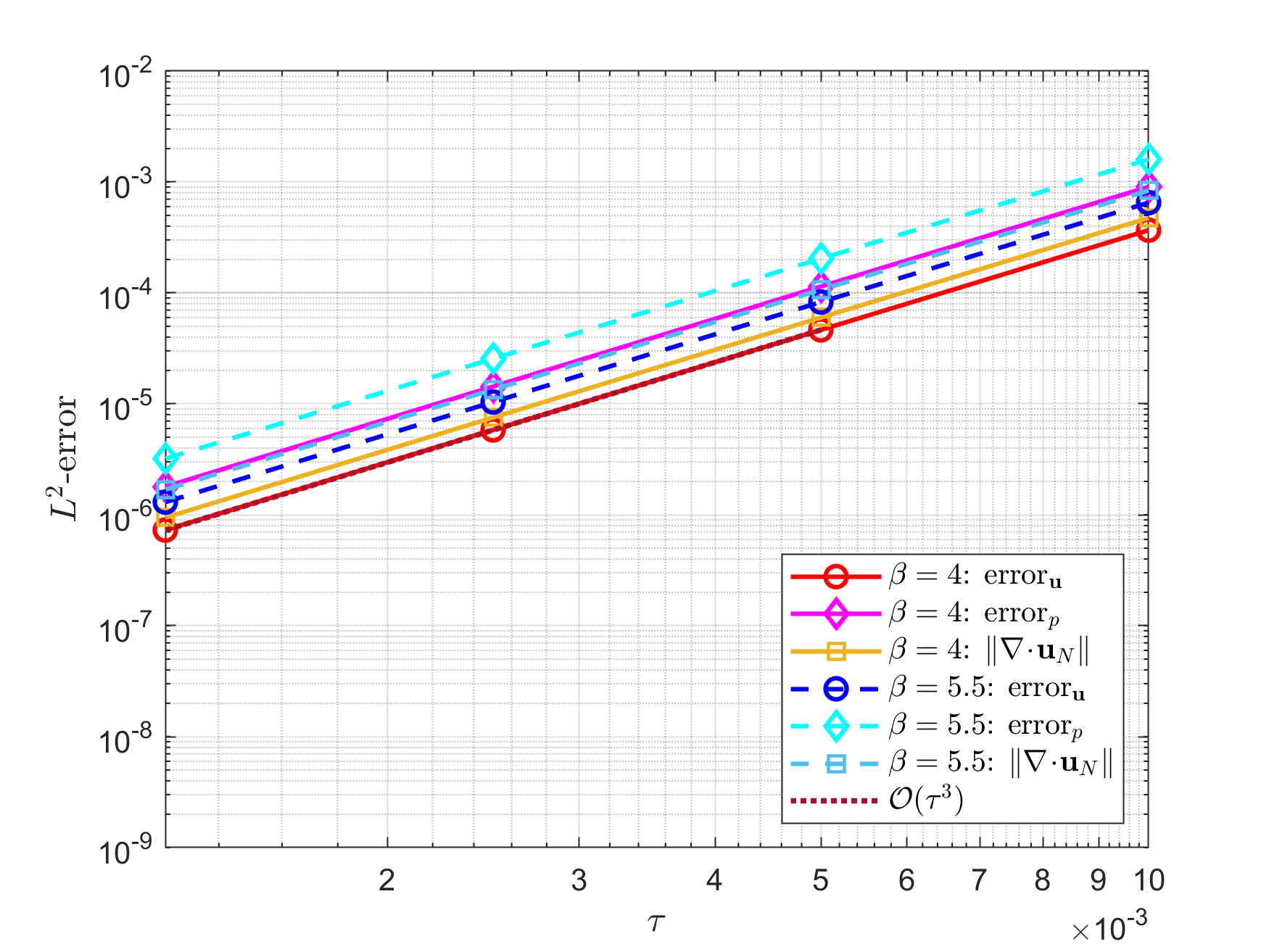}}\\
		\subfigure[GBDF-4]{\includegraphics[width=0.42\textwidth]{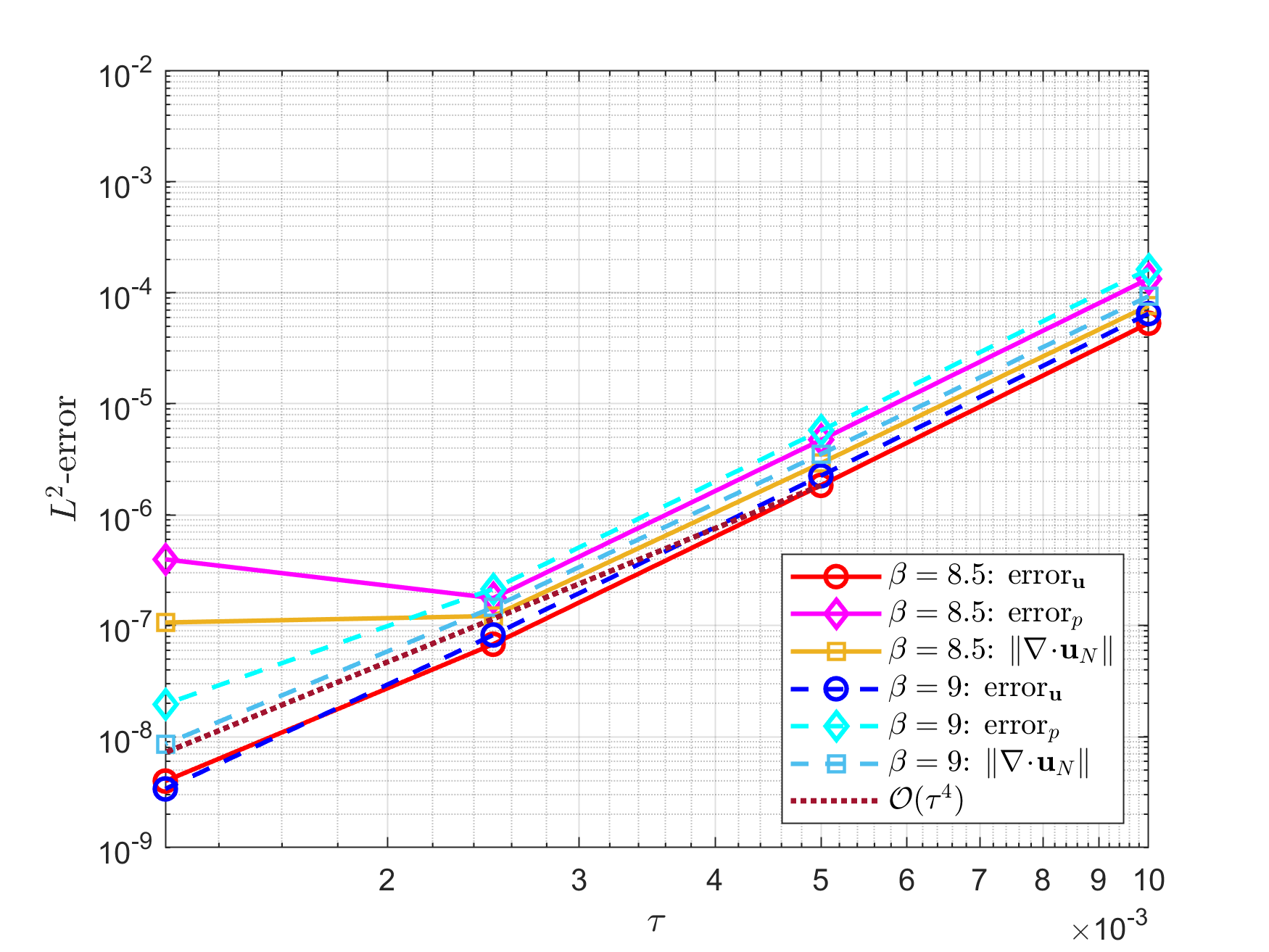}}
		\subfigure[GBDF-5]{\includegraphics[width=0.42\textwidth]{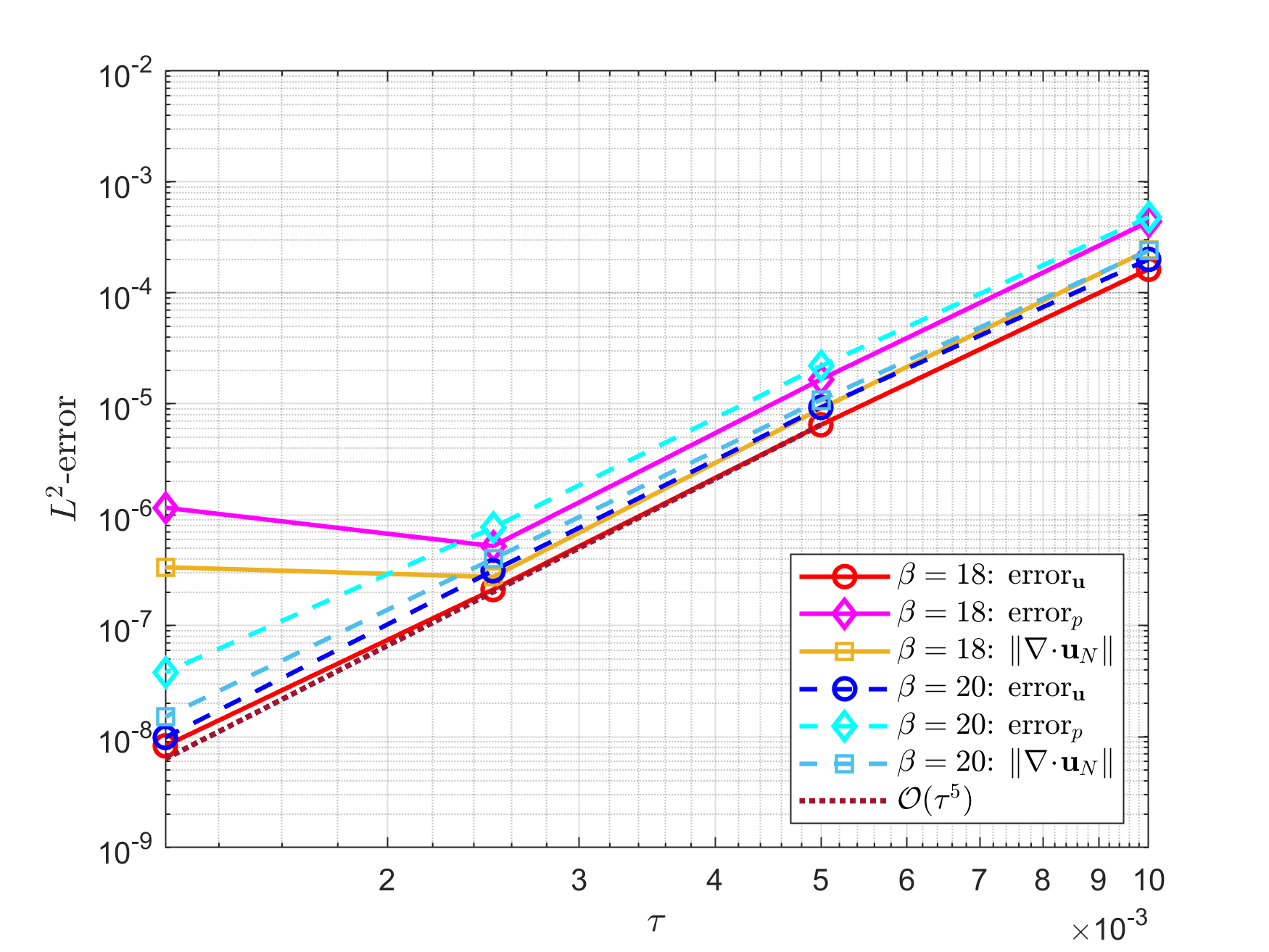}}\\
		\caption{Convergence test for the GBDF-$\mathrm{k}$ consistent splitting schemes.}\label{fig:conv_GBDF}
	\end{figure}
	
	\begin{figure}[htb!]
		\centering
		\subfigure[SIELM-4]{\includegraphics[width=0.32\textwidth]{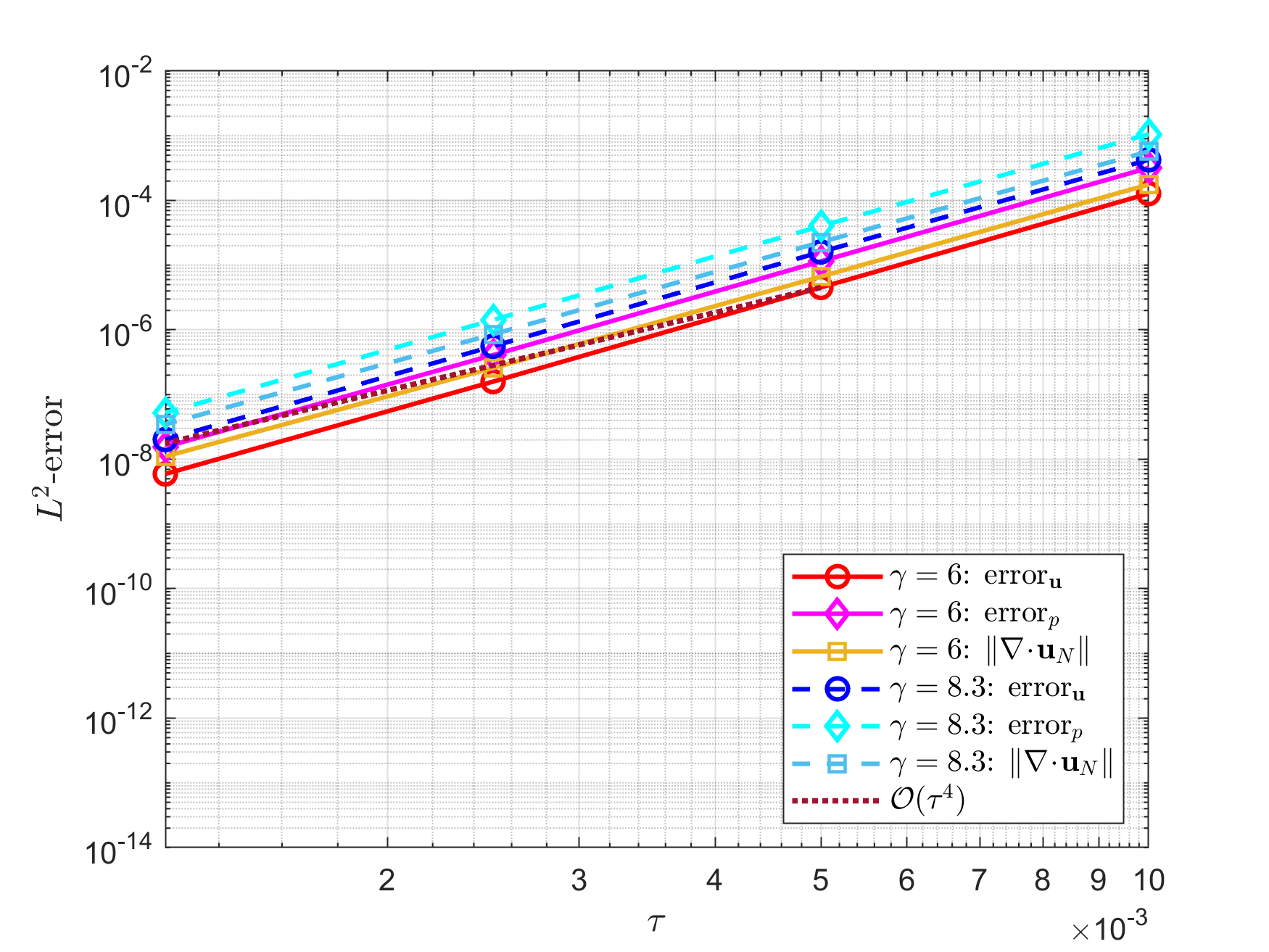}}
		\subfigure[SIELM-5]{\includegraphics[width=0.32\textwidth]{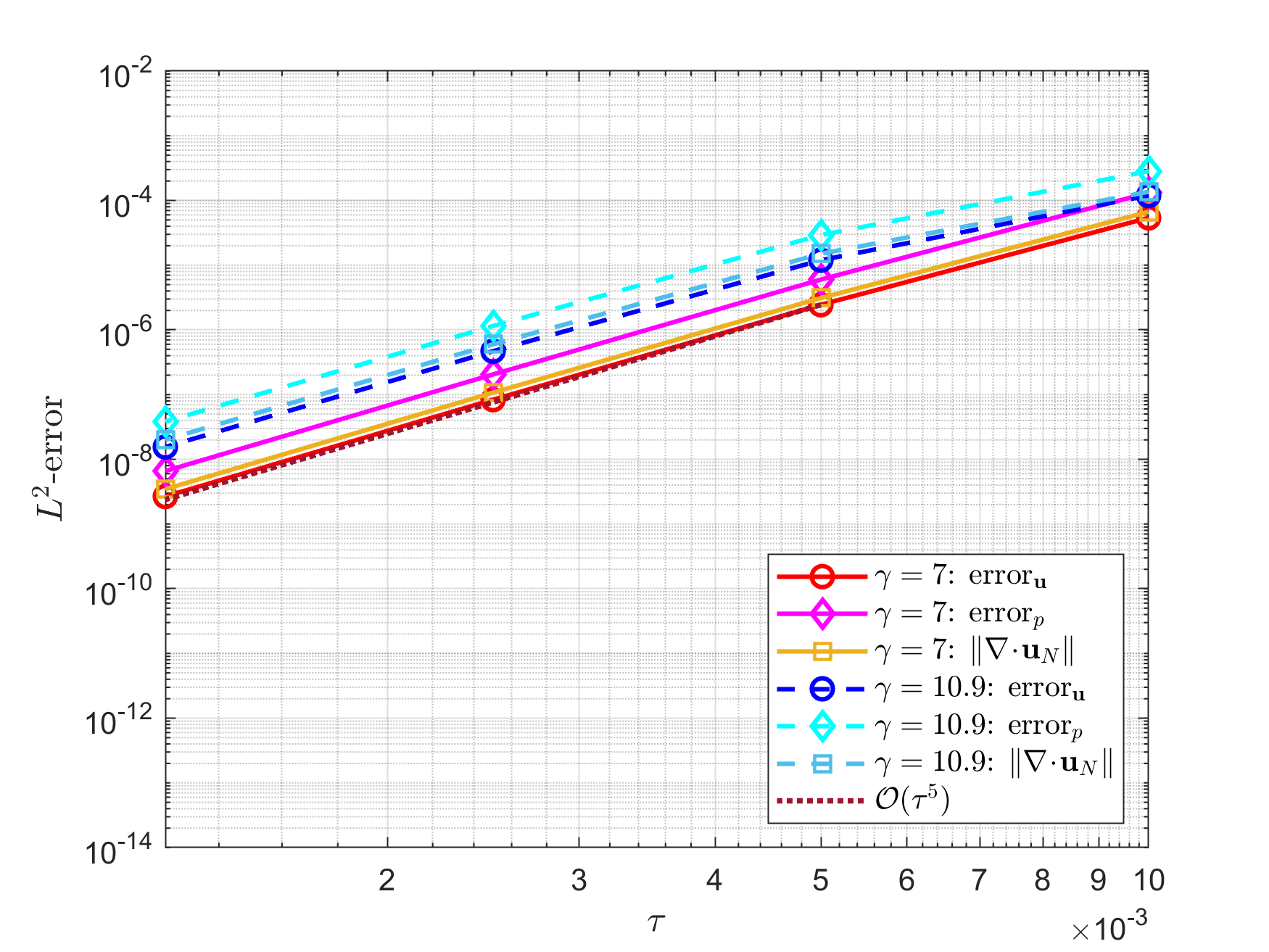}}
		\subfigure[SIELM-6]{\includegraphics[width=0.32\textwidth]{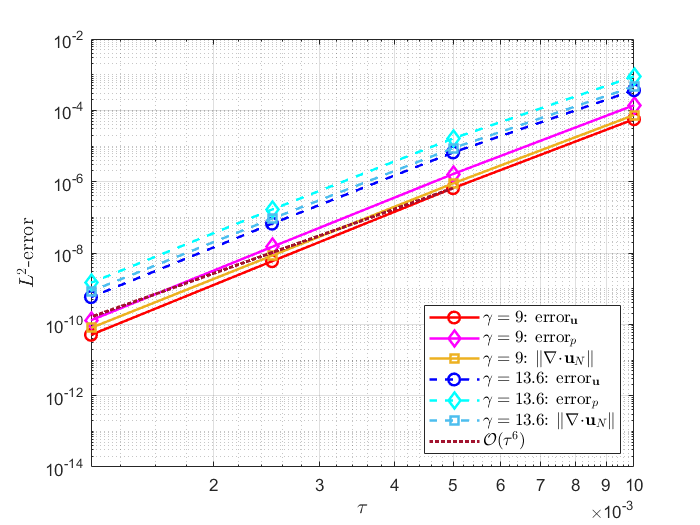}}\\
		\subfigure[SIELM-7]{\includegraphics[width=0.32\textwidth]{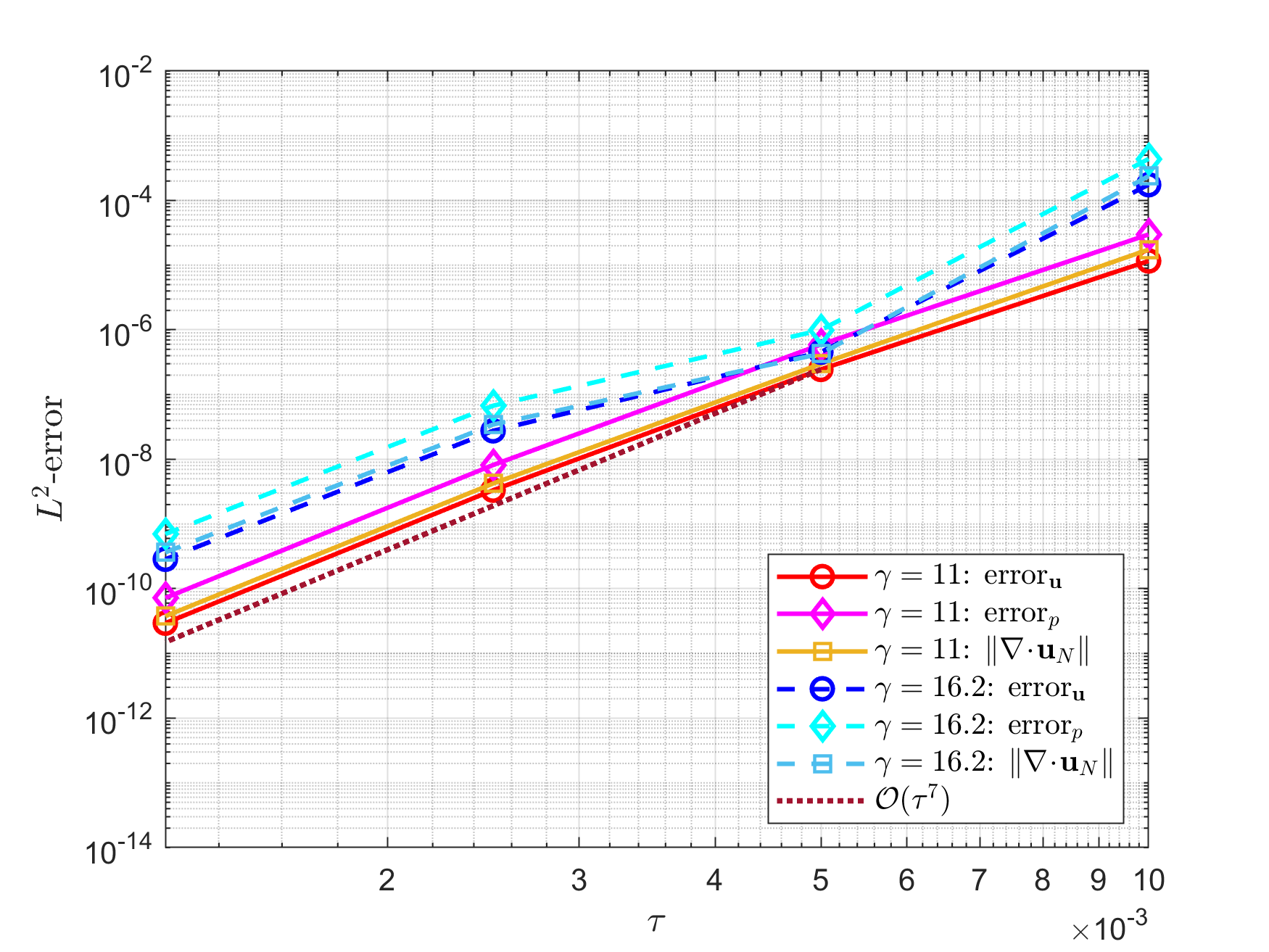}}
		\subfigure[SIELM-8]{\includegraphics[width=0.32\textwidth]{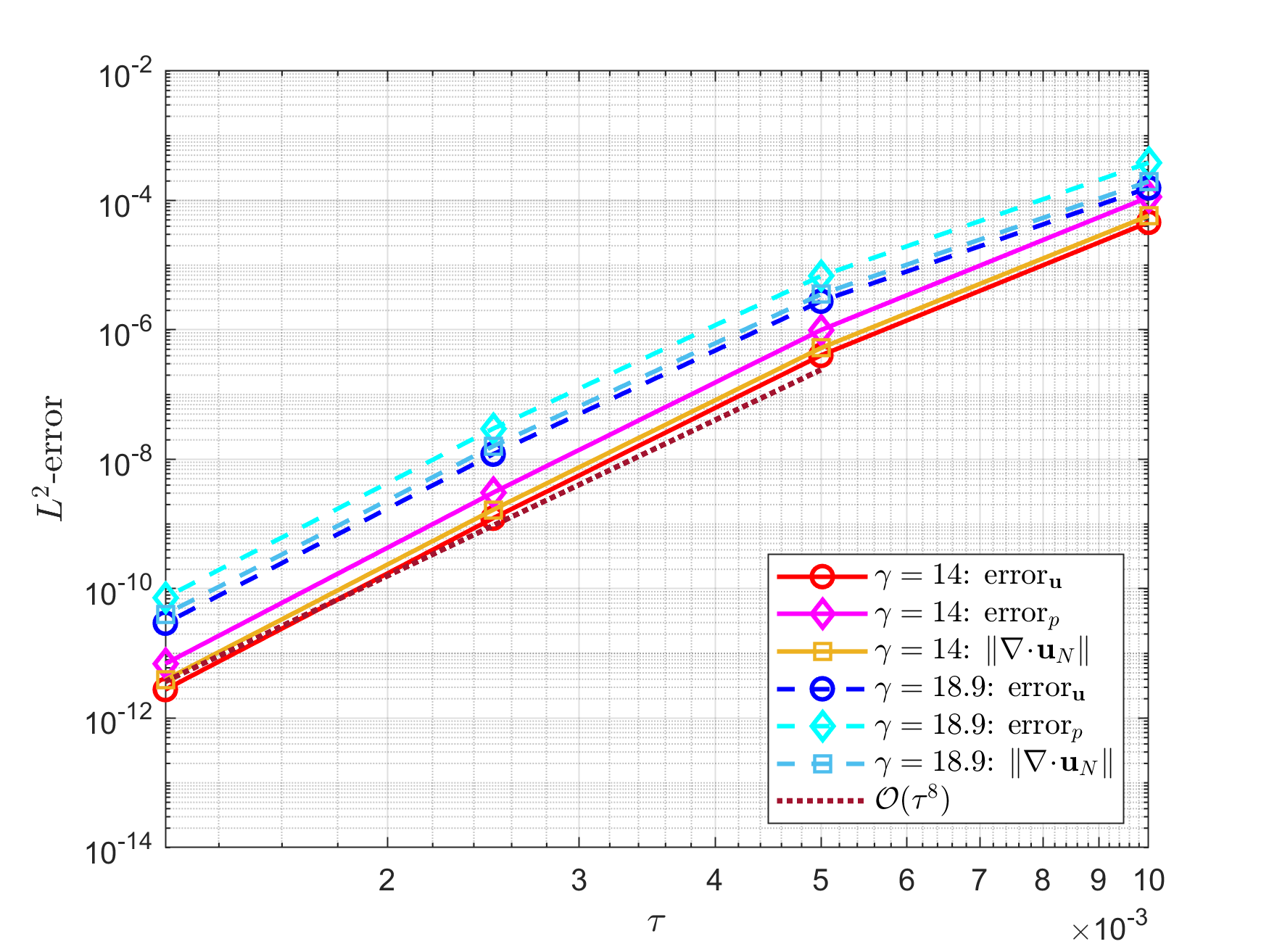}}
		\subfigure[SIELM-9]{\includegraphics[width=0.32\textwidth]{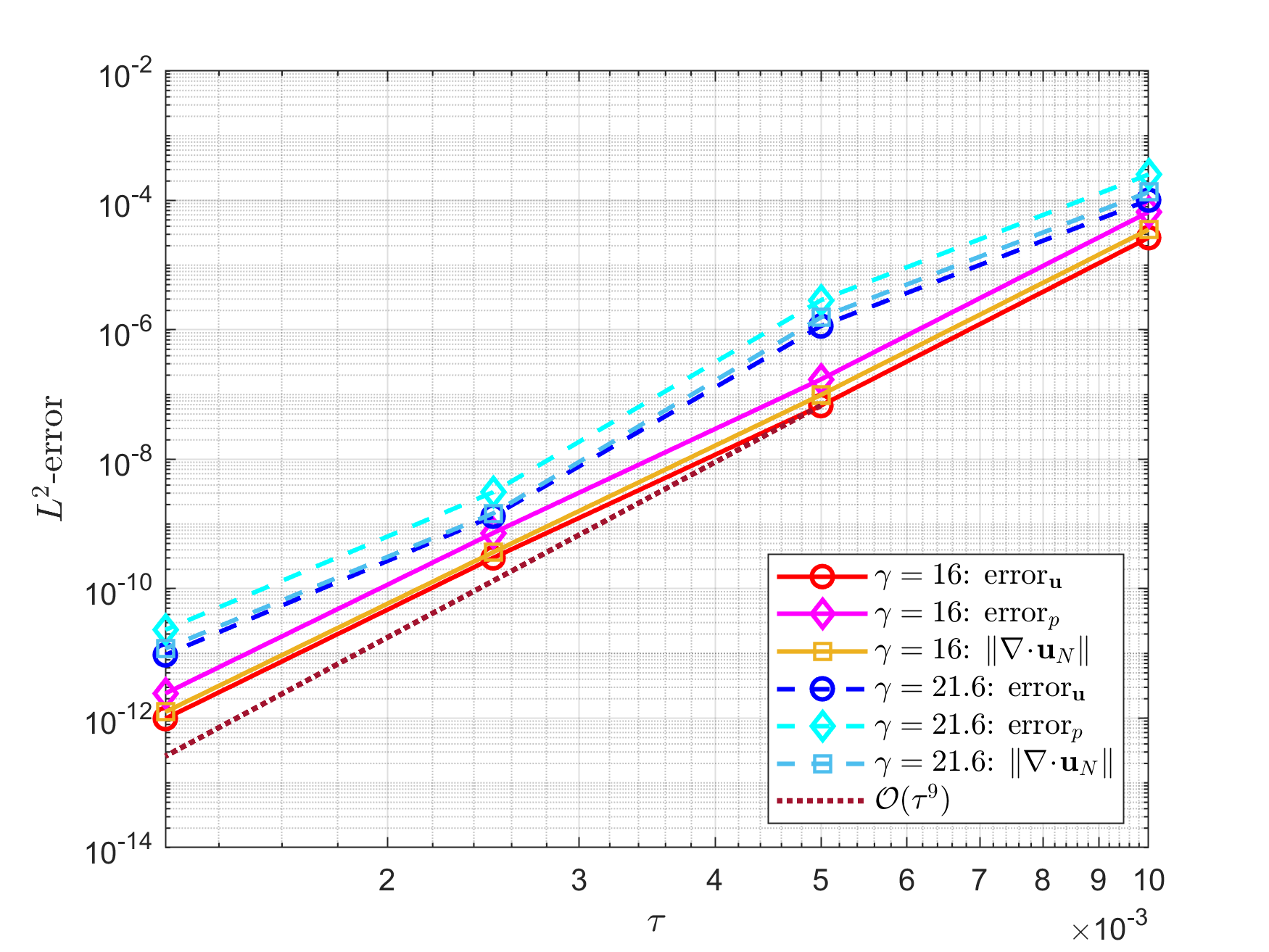}}
		\caption{Convergence test for the SIELM-$\mathrm{k}$ consistent splitting schemes.}\label{fig:conv_SIELM}
	\end{figure}

	Fig.~\ref{fig:stokes_SIELM} presents the numerical results for the SIELM3 and SIELM4 schemes with the corresponding theoretical requirements $\gamma > 5.56667$ and $\gamma>8.22174$ for the unconditional stability, respectively. It is observed that the SIELM3 scheme is stable for $\gamma\ge2$ and the SIELM4 scheme is stable for $\gamma\ge6$. It seems that the proposed stability indicator $\mathfrak{I}_{\mathrm{IE},\mathrm{S}}^{\mathrm{(k)}}>\sqrt{2}/2$ is sufficient but always not sharp.

	We next test the SIELM3 and SIELM4 schemes for the unconstrained INSE system \eqref{cont: unconstrained INSE-pressure Poisson}-\eqref{cont: unconstrained INSE-momentum} with $\nu=0.005$ and the zero-valued external force $\myvec{f}=\myvec{0}$. The reference solution is computed using the SIELM4 scheme ($\gamma=9$) on a refined space mesh with $N_x=N_y=192$. Fig.~\ref{fig:vorticity} displays the energy evolution for the SIELM3 and SIELM4 schemes with $\tau=0.0005$. As predicted in Theorem \ref{thm: NS multistep stability}, the solutions generated by the two certified SIELM schemes are stable and agree well with the reference solution. Also, the SIELM3 scheme with $\gamma=3$ and the SIELM4 scheme with
	$\gamma=6$ remain stable in our tests, although these parameter values are smaller than the theoretically certified value.
	
	The qualitative behaviors of numerical solutions, see the vorticity contours at $T=0.01$, $T=3$ and $T=5$, are depicted in Fig.~\ref{fig:vorticity}. As seen, the vortex structures produced by the SIELM schemes agree well with the reference solution, confirming that the stable energy behavior is accompanied by physically meaningful flow patterns.

	
	\begin{figure}[htbp]
		\centering		
		\subfigure[SIELM-4 $(\gamma=8.3)$]{\includegraphics[width=0.32\textwidth]{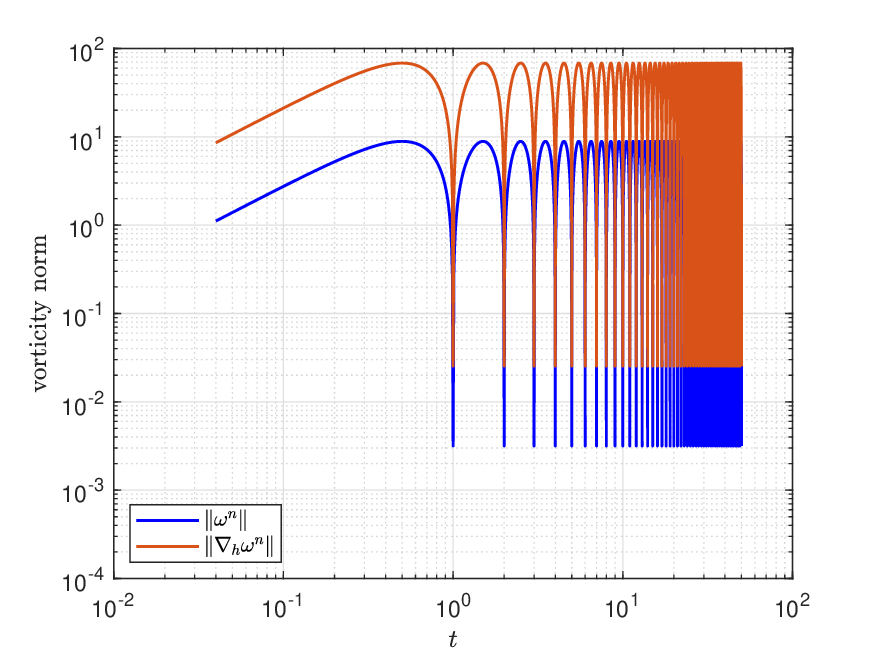}}
		\subfigure[SIELM-5 $(\gamma=10.9)$]{\includegraphics[width=0.32\textwidth]{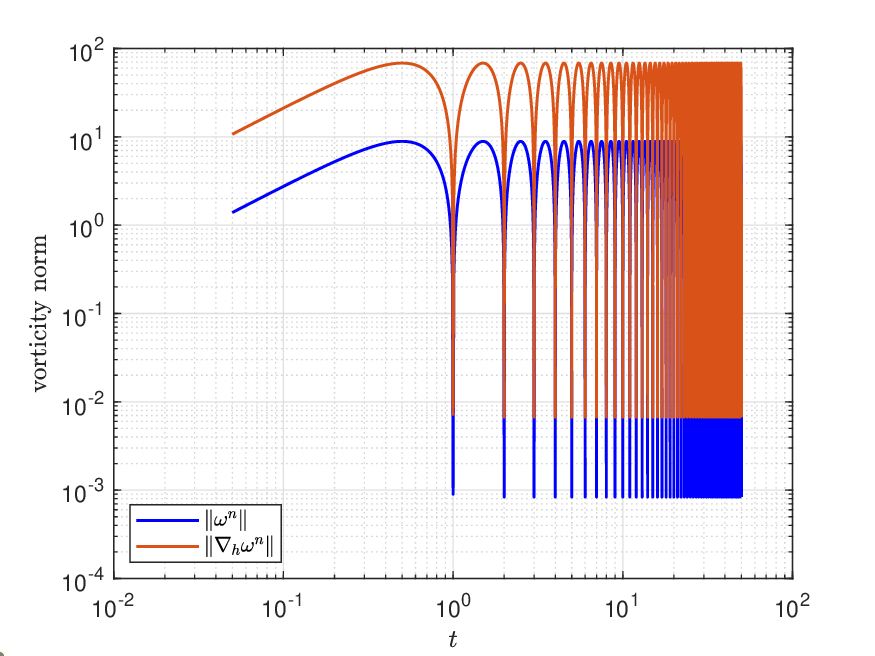}}
		\subfigure[SIELM-6 $(\gamma=13.6)$]{\includegraphics[width=0.32\textwidth]{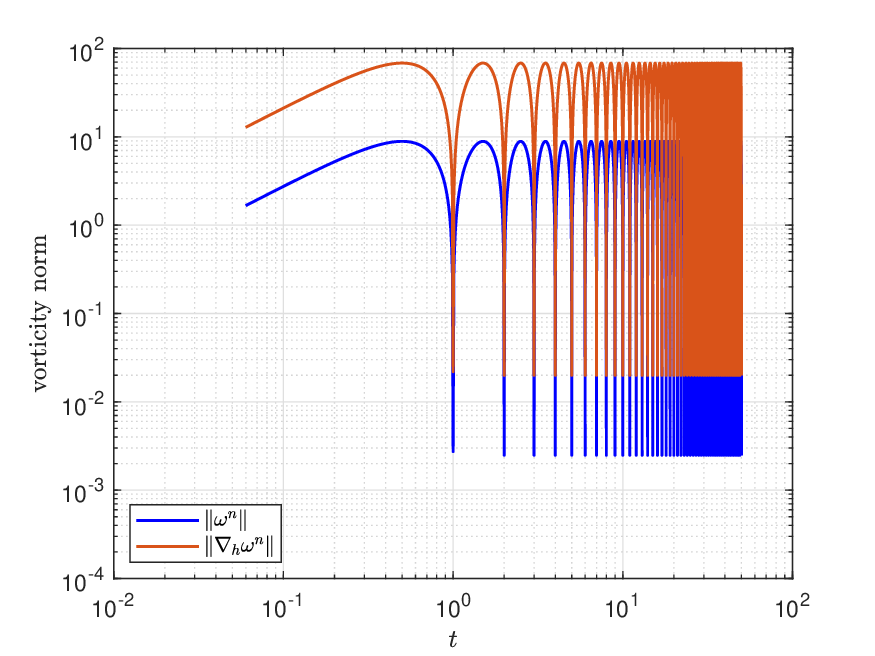}}\\
		\subfigure[SIELM-7 $(\gamma=16.2)$]{\includegraphics[width=0.32\textwidth]{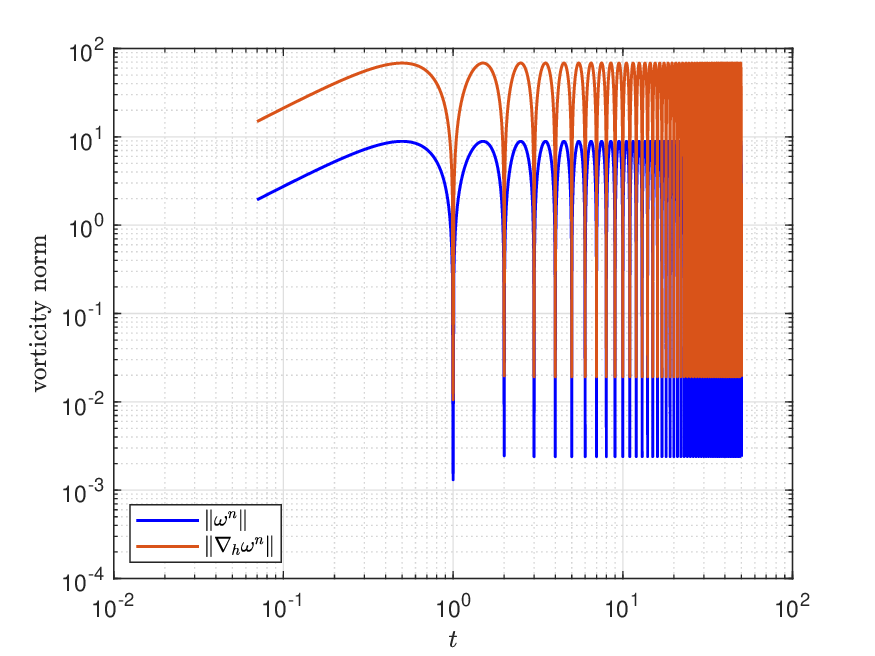}}
		\subfigure[SIELM-8 $(\gamma=18.9)$]{\includegraphics[width=0.32\textwidth]{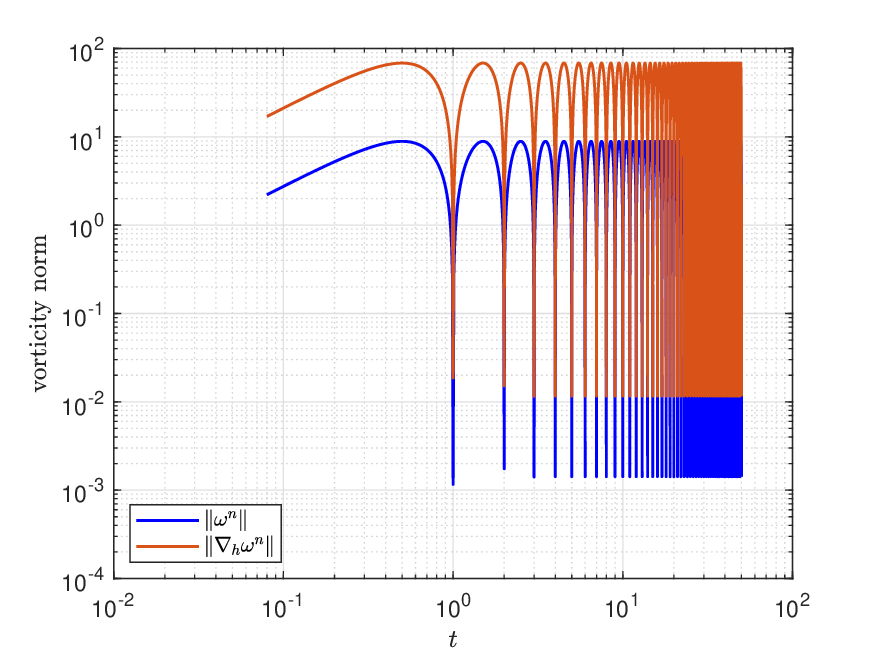}}
		\subfigure[SIELM-9 $(\gamma=21.6)$]{\includegraphics[width=0.32\textwidth]{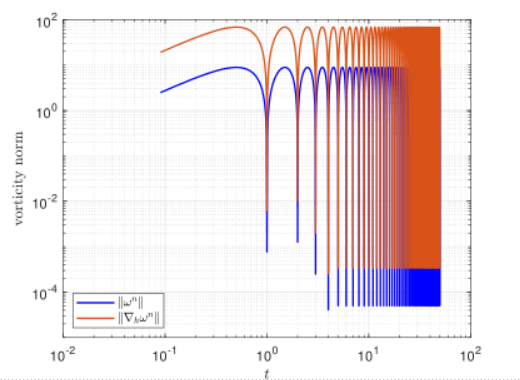}}
		\caption{The vorticity in the discrete $L^2$ and $H^1$ norms for the SIELM-$\rmk$ schemes.}\label{fig:vorticity_norm_SIELM}
	\end{figure}
	
	\begin{example}\cite[Example 2]{HuangShen:2025mcom}
		Consider the  INSE model \eqref{cont: INSE-momentum}-\eqref{cont: INSE-incompressible} on the domain $\Omega=(-1,1)^2$ with no-slip boundary condition, which admits the exact solution
		$$\myvec{u}=\bra{\sin2\pi y\sin^2\pi x\sin \pi t,-\sin2\pi x\sin^2\pi y\sin\pi t},$$
		and $p=\cos\pi x\sin\pi y\sin \pi t$. We set the viscosity coefficient $\nu=1$ and use the Legendre-Galerkin  method with $Nx=Ny=64$ modes in space so that the spatial discretization error is negligible compared with the time discretization error. \end{example}

	We now verify the temporal convergence rates of the GBDF-$\rmk$ ($2\le \rmk\le5$) and SIELM-$\rmk$ ($4\le \rmk\le9$) schemes. Always, we choose two different parameters $\beta$ (or $\gamma$) for each method, including a theoretically certified value from the stability condition $\mathfrak I_{\mathrm{IE},\mathrm{S}}^{(k)}>\sqrt{2}/2$ and a smaller one for examining the practical effect of the parameters on the time accuracy. More precisely, for the GBDF-$\rmk$ schemes, we test GBDF2 with $\beta=1,3$, GBDF3 with $\beta=4,5.5$, GBDF4 with $\beta=8.5,9$, and GBDF5 with $\beta=18,20$. For the SIELM-$\rmk$ schemes, \lan{we test SIELM4 with $\gamma=6,8.3$, SIELM5 with $\gamma=7,10.9$, SIELM6 with $\gamma=9,13.6$, SIELM7 with $\gamma=11,16.2$, SIELM8 with $\gamma=14,18.9$ and SIELM9 with $\gamma=16,21.6$.}  In each pair, the larger parameter is certified by our stability condition, while the smaller one is always not covered by the current theory.

	Fig.~\ref{fig:conv_GBDF} shows the log-log plots of the $L^{2}$ norm errors of the velocity and pressure, and the $L^{2}$ norm $\|\nabla\cdot\myvec{u}_N\|$ of discrete divergence at the final time $T=1$ for the GBDF-$\rmk$ schemes with $2\le \rmk\le5$ against the varied time steps $\tau=10^{-2}/2^m$ $(0\le m\le 3)$. For all tested values of $\beta$, the GBDF-$\rmk$ schemes generate approximately the $\rmk$-th order convergence in time, as predicted by Theorem \ref{thm: NS multistep convergence}. It seems that the GBDF-$\rmk$ schemes with the smaller parameter can retain the desired temporal accuracy if they practically maintain the numerical stability.
	Fig.~\ref{fig:conv_SIELM} presents the corresponding results for the SIELM-$\rmk$ schemes with $4\le\rmk\le9$. Again, for all tested values of $\gamma$, the observed slopes match the designed temporal accuracy $O(\tau^\rmk)$ for the SIELM-$\rmk$ schemes. As seen, the smaller choices of $\gamma$ also retain the expected convergence rates in these tests provided that the computation remains stable.
	
	We observe that, for the same index $\rmk$, a smaller value of $\beta$ or $\gamma$ usually leads to smaller $L^2$ errors for both the velocity and the pressure. They are consistent with the fact that a larger parameter improves the stability but always introduces a larger value of truncation error, cf. \cite[Sections SM2 and SM3]{LiaoQuanTangZhou:IMES}. Therefore, a moderately smaller $\beta$ or $\gamma$ would be practically desirable in numerical simulations; however, up to now, we are not able to improve the stability condition $\mathfrak I_{\mathrm{IE}}^{(k)}>\sqrt{2}/2$.

\lan{We further examine the long-time behavior of the SIELM-$\rmk$ schemes by monitoring the vorticity profile. For the exact solution in the above example, the corresponding vorticity is
\begin{equation*}
  \omega=\partial_x \myvec{u}_2-\partial_y \myvec{u}_1
  =-2\pi\left[\cos(2\pi x)\sin^2(\pi y)+\cos(2\pi y)\sin^2(\pi x)\right]\sin(\pi t). 
\end{equation*}
Therefore, $\omega(\cdot,t)=0$ and $\nabla\omega(\cdot,t)=\myvec{0}$ whenever $t\in\mathbb{N}$. This provides a useful benchmark for the long-time simulation. On the one hand, the long-time boundedness of the discrete vorticity norms reflects the stability of the numerical scheme.  On
the other hand, since the exact vorticity vanishes at every integer time, the numerical vorticity should also exhibit pronounced decay near these time levels if the scheme accurately captures the temporal oscillation of the solution. We take the time step size $\tau=0.01$ and compute the numerical solution up to $T=50$. At each time level, we compute the discrete $L^2$ norm and $H^1$ semi-norm of the vorticity for the numerical solution $\myvec{u}_N^n$.  Fig.~\ref{fig:vorticity_norm_SIELM} shows the time histories of these two quantities for the SIELM-$\rmk$ schemes with $4\le\rmk\le9$.  The curves remain uniformly bounded over the long time interval $[0,50]$, which indicates the long-time stability of the SIELM schemes. }

\end{document}